%% file: CLT1d_v4.tex
\documentclass[a4paper,11pt]{article}
\usepackage[T1]{fontenc}
\usepackage[utf8x]{inputenc} 
\usepackage{hyperref} \hypersetup{colorlinks=true, citecolor=blue, linkcolor=black}
\usepackage{lmodern}
\usepackage{amsmath}
\usepackage{amsthm}
\usepackage{amssymb}
\usepackage{authblk}
\usepackage{graphicx}
\usepackage{enumitem}
\usepackage{xcolor}
\usepackage{dsfont}
\usepackage{mathrsfs}
\include{thm-config}

\usepackage{dsfont}
\usepackage{todonotes}
\newcommand{\1}{\mathds{1}}
\renewcommand{\i}{\mathbf{i}}

\usepackage{a4wide}
\title{Central limit theorem for smooth statistics of one-dimensional free fermions}
\author{Alix
Deleporte\thanks{alix.deleporte@universite-paris-saclay.fr\\
A.D. acknowledges the support of the PEPS 2021 CNRS grant.},
Gaultier Lambert\thanks{glambert@kth.se \\ G.L. acknowledges the support of the Ambizione grant S-71114-05-01 from the Swiss National Science Foundation (University of Zurich) and of the VR starting grant 2022-04882 from the  Swedish Research Council.}}
\affil{Universit\'e Paris-Saclay, CNRS, Laboratoire de math\'ematiques d'Orsay, 91405, Orsay, France.}
\affil{KTH Royal Institute of Technology, Department of Mathematics, 11428, Stockholm, Sweden.}

\newcommand{\X}{\mathrm{X}}
\newcommand{\g}{\mathrm{g}}

\usepackage{scrtime}

\begin{document}

\maketitle

\begin{abstract}
We consider the determinantal point processes associated with the spectral
projectors of a Schrödinger operator on $\R$, with a smooth confining potential. In the semiclassical limit, where the number of particles tends to infinity, 
we obtain a Szeg\H{o}-type central limit theorem (CLT) for the
fluctuations of smooth linear statistics. More precisely, the Laplace
transform of any statistic converges without renormalization to a Gaussian limit with a $H^{1/2}$-type variance, which depends on the potential. 
In the one-well (one-cut) case, using the quantum action-angle theorem and additional micro-local tools, we reduce the problem to the asymptotics of Fredholm determinants of certain approximately Toeplitz operators. 

In the multi-cut case, we show that for generic potentials, a similar
result holds and the contributions of the different wells are independent in the limit.
\end{abstract}

\section{Introduction}
\label{sec:introduction}

\paragraph{Free fermions.}
This article is a follow-up on \cite{deleporte_universality_2021} on the semiclassical
analysis of free fermions or determinantal point processes associated with spectral projectors of Schrödinger operators, \cite{Macchi75}.
To define the model, let $V\in C^{\infty}(\R,\R)$ be a function such that $V(x)\to +\infty$ as $x\to \pm \infty$.
Then, for $\hbar\in(0,1]$, the operator
\( 
-\hbar^2\Delta+V
\)
is essentially self-adjoint with compact resolvent on $L^2(\R)$ \cite{helffer_spectral_2013}.
Consequently, given $\mu\in \R$, the orthogonal projector
\begin{equation} \label{proj}
\Pi_{\hbar}(\mu)=\1_{(-\infty,\mu]}(-\hbar^2\Delta+V)
\end{equation}
is well-defined by the spectral theorem, and has finite rank $N=N_\hbar(\mu)$.
This operator defines a determinantal point process on $\R$, denoted
by $\X : = \sum_{j=1}^N \delta_{x_j}$, which describes
  the joint spatial distribution of $N$ fermions occupying the
  low-energy states of $-\hbar^2\Delta+V$. This point process is characterized by the property that for any function $f\in L^{\infty}(\R,\C)$, the Laplace transform of the linear statistic $\X(f)$ is given by the Fredholm determinant
\begin{equation} \label{Laplace}
\mathbb{E}[e^{\X(f)}]=\det(\mathrm{I}+(e^f-1)\Pi_{\hbar}) 
\end{equation}
where the function $e^{f}-1$ is interpreted as a bounded operator on $L^2(\R)$ acting by multiplication.
We refer to the surveys \cite{Soshnikov_01,hough_determinantal_2006} and the appendix of \cite{deleporte_universality_2021} for different reviews about determinantal processes. 
Such free fermions point processes have also been extensively studied
in the physics literature,
e.g.~\cite{eisler_universality_2013,dean_universal_2015,dean_noninteracting_2019,dean_nonequilibrium_2019,dean_impurities_2021,Smith_21}
and for the harmonic oscillator in \cite{akemann_elliptic_2023}.

\smallskip

These point processes can be studied for a general smooth (confining) potential $V$ in the semiclassical limit $\hbar\to0$.
By Weyl's law, the number of particles $N\sim \frac{1}{\pi\hbar}\int (\mu-V)_+^{1/2}$ tends to infinity while the point process concentrates on a deterministic measure supported on the set $\{V\le \mu\}$, which is sometimes called \emph{the droplet}, see e.g.~\cite[Theorem I.1]{deleporte_universality_2021}. 
The main goal of \cite{deleporte_universality_2021} was to establish
that the (microscopic) scaling limits of these point processes in the
bulk and at the edge of the droplet of particles are given by the Sine, respectively the Airy, point processes from random matrix theory \cite{Tao12,dean_noninteracting_2019}. 
These results rely on the asymptotics of the kernel associated to $\Pi_{\hbar}$ and they are generalized to free fermions on $\R^n$ for any $n\ge 1$; the scaling limits are independent of the potential $V$. 
There is, however, a major difference in the behavior of linear statistics in dimension 1 versus $n\ge 2$.
Indeed, one expects that for any (non-constant) smooth function $f:\R^n\to\R$ with suitable growth, the corresponding linear statistic
satisfies $\operatorname{var} \X(f) \asymp \hbar^{1-n} $ as
$\hbar\to0$. It is an open problem to obtain exact asymptotics for these variance in dimension $n\ge 2$. We proved in \cite{deleporte_universality_2021} that if $f\in C^\infty_c(\{V<\mu\},\R)$, then $\operatorname{var} \X_\hbar(f) \le C(f) \hbar^{1-n}$, while for any (non-constant) function $f\in H^1(\R^n)$, we have
$\hbar^{2-n}\operatorname{var} \X_\hbar(f)  \to\infty $ as
$\hbar\to0$. This lower-bound is not sharp, but it does imply a CLT
for linear statistics, cf.~\cite[Theorem
I.3]{deleporte_universality_2021} because the variance tends to
infinity. In contrast, in dimension $n=1$, the variance stays bounded
and a more precise study is needed to establish a CLT, if any.

\paragraph{Main results.}
In this article, we therefore focus on one-dimensional free fermions
point processes and our goal is to show that for any test function
$f\in C^\infty_c(\R,\R)$, the fluctuations of the linear statistics
$\X(f)$ are Gaussian in the limit $\hbar\to 0$; we also determine the limit of
$\operatorname{var} \X(f)$ and how it depends on the Newtonian dynamics for the
potential $V$ at energy $\mu$. To this end, we directly analyse the determinant
\eqref{Laplace} by using the \emph{quantum action-angle theorem}
to reduce the problem to the harmonic oscillator; this
procedure is specific to the 1D case. This strategy works directly in
the \emph{one-cut} case where $\{V\le \mu\}$ is connected; we can also
treat \emph{multi-cut} situations under some generic assumptions.

\begin{theorem}[One-cut case]\label{thr:single_well}
Let $V\in C^{\infty}(\R,\R)$, with $V(x)\to +\infty$ as $x\to \pm
\infty$. Let $\mu\in\R$ be such that $\{V = \mu\} =  \{x_0^-,x_0^+\}$ with $x_0^-<x_0^+$ and $V'(x_0^\pm)\neq 0$. 
Then, for any $f\in C^{\infty}_c(\R,\R)$, uniformly for
$\eta\in\C$ sufficiently small,
\[
\log\det(\mathrm{I}+(e^{\eta f}-1)\Pi_{\hbar})=\eta\tr(f\Pi_{\hbar})
+ \frac{\eta^2}2 \Sigma_{(V,\mu)}^2(f) + \underset{\hbar\to0}{o(1)}. 
\]
This implies that in distribution, as the number of particles $N\to\infty$, the random variable
\begin{equation} \label{distcvg}
\X(f) - \mathbb{E} \X(f) \, \to\,  \Sigma_{(V,\mu)}(f)\,\mathcal{N}
\end{equation}
where $\mathcal{N}$ is a standard Gaussian. 
The limiting variance is a weighted $H^{1/2}$-seminorm, meaning that there is a map $\psi : [0,2\pi] \to [x_0^-,x_0^+]$ so that 
\begin{equation} \label{var}
\Sigma_{(V,\mu)}^2(f) = \sum_{k\in\N} k \bigg| \int_0^{2\pi} e^{\i k \theta} f(\psi(\theta)) \frac{\dd\theta}{2\pi} \bigg|^2 . 
\end{equation}
The map $\psi$ depends on the Newtonian dynamics with potential $V$ at
energy $\mu$ and it is constructed in Section~\ref{sec:notations}. (See also Section~\ref{sec:GFF} for an interpretation of \eqref{var} in terms of the Gaussian free field.)
\end{theorem}

By a standard truncation, a version of Theorem~\ref{thr:single_well}
holds for continuous test functions with at most exponential growth
and which are $C^{\infty}$ in a neighborhood of the droplet $\{V\le
\mu\}$. In this case, the distributional limit \eqref{distcvg} still
holds and all centred moments of the linear statistic $\X(f)$ also
converge; we clarify this in Proposition~\ref{prop:clt}.
Formula \eqref{var} can be recovered from the asymptotics of the covariance kernel of the counting function $\R\ni x\mapsto \X(\1_{(-\infty,x]})$ of the fermions point process. These asymptotics are well-known for the harmonic oscillator (see \cite{CFLW21} for a thorough study of the GUE counting function) and have been derived using physical arguments in \cite{Smith_21}, for general one-cut potential, based on WKB asymptotics of the Schr\"odinger operator eigenfunctions.  
We will report on the relationship with \cite{Smith_21} and the Gaussian free field interpretation of Theorem~\ref{thr:single_well} in the Appendix~\ref{sec:GFF}.

\begin{example} \label{ex:gue}
For the \emph{quantum harmonic oscillator}, $V(x) = x^2$ on $\R$ and $\mu=1$, the Schr\"odinger operator $-\hbar^2\Delta+V$ is  diagonalized by the Hermite functions with respect to the weight $x\mapsto e^{-x^2/\hbar}$. 
Then, the free fermions point process corresponds to the eigenvalues
of the Gaussian Unitary Ensemble (GUE) and the equilibrium measure is the
Wigner semi-circle law $\rho(x) =
\frac{2}{\pi}\sqrt{1-x^2}$. Moreover, according to
Section~\ref{sec:notations}, one has $\psi(\theta) = \cos\theta$ for $\theta\in[0,2\pi]$, so that 
$\Sigma^2(f) = \sum_{k\in\N} k |\hat{f}_k|^2$ where $(\hat{f}_k)_{k\in\N_0}$ denotes the Fourier-Chebyshev coefficients of $f$. 
In this case, Theorem~\ref{thr:single_well} corresponds to
\cite[Thm~2.4]{Johansson_98} with $\beta=2$; see also ~\cite{Tao12}, Chap.~3. Note that in this limit, one can also replace the mean $\mathbb{E} \X(f)$ by $N \int f \dd\rho$ up to a vanishing error.
\end{example}

The results of \cite{Johansson_98} hold for general $\beta$-ensembles
(or 1d log-gases) in the one-cut case and the strategy there relies on the so-called loop equation method. 
In contrast to Theorem~\ref{thr:single_well}, the asymptotic variance for one-cut $\beta$-matrix models is universal in the sense that it depends on the potential $V$ only via the support of the equilibrium measure.
We refer to \cite{lambert_quantitative_2019,bekerman_clt_2018,bourgade_optimal_2022} and the last version of \cite{borot_asymptotic_2013-2} for further recent developments on the loop equation method for $\beta$-ensembles.

It is also of interest to compare Theorem~\ref{thr:single_well} to the (strong) \emph{Szeg\H{o} limit theorem} which arises in a slightly different geometric context, considering free fermions on a  one-dimensional torus (CUE). 

\begin{example} \label{ex:cue}
Consider the determinantal process associated with the operator $
\1_{(-\infty,1]}(-\hbar^2\Delta)$ where $\Delta$ is the Laplacian on
$\mathbb{T} = [0,2\pi]$ with periodic boundary conditions so that the
number of particles is $N = 2\hbar^{-1}+1$. 
The eigenfunctions of this operator are just Fourier modes and this point process corresponds to the eigenvalues of the classical \emph{circular unitary ensemble} (CUE) from random matrix theory.
In particular, one can rewrite \eqref{Laplace} as
\[
\mathbb{E}[e^{\X(f)}]=\det(A^{(N)}) 
\]
where $A$ is a Toeplitz matrix, $\big(A^{(N)}_{ij} = \widehat{g}_{i-j} \big)_{i,j \in [N]}$, and $\big(\widehat{g}_{k} \big)_{k\in\Z}$ denotes the usual Fourier-coefficient of  the function $g=e^{f}-1$. 
The asymptotics of such \emph{Toeplitz determinants} are a classical subject in analysis and are known as Szeg\H{o} limit theorems. 
In particular, if $f\in H^{1/2}(\mathbb{T},\C)$ (that is if $\sum_{k\in\Z} \sqrt{1+k^2} |\widehat{f}_{k}|^2 <\infty$), then 
\[
\log\det(A^{(N)}) =  N \widehat{f}_{0} + \frac12 \sum_{k\in\Z} |k| \widehat{f}_k  \widehat{f}_{-k} + \underset{N\to\infty}{o(1)}. 
\]
This result is discussed in details in \cite[Chap.6]{simon_orthogonal_2005} with several different proofs, including the combinatorial approach developed by Kac \cite{Kac54} and Soshnikov \cite{Soshnikov_00} that we follow in Section~\ref{sec:end-proof}. 
We also refer to the survey \cite{deift_toeplitz_2013} for further
proofs and applications of the strong Szeg\H{o} limit theorem and to
\cite{charlier_asymptotics_2021,lambert_strong_2023} for the relationship with the GUE eigenvalue fluctuations and log-correlated Gaussian fields.

\end{example}

In the \emph{multi-cut} case, that is when the droplet $\{V\leq \mu\}$
has several
disconnected components, we can also formulate a CLT under an
additional genericity condition on the potential. This genericity
condition ensures that the different cuts are non-resonant (their eigenvalues are sufficiently far away from each other),
so that eigenfunctions are all localised in exaclty one of the cuts. This genericity
holds along any discrete set of values of $\hbar$ accumulating at $0$ with
at least polynomial speed. Such subsequences are 
relevant in the context of quantization on a compact phase
space (where $\hbar_m\asymp 1/m$); they are also relevant since the
number of particles is expected to grow linearly (since $\hbar N$ is
roughly constant by the Weyl law).

\begin{theorem}[Multi-cut case] \label{thr:multiple_wells}
Let $V\in C^{\infty}(\R,\R)$ with $V(x)\to +\infty$ as $x\to \pm
\infty$, and assume that $\mu\in\R$ satisfies, for some $\ell\ge 2$, 
\[
\{V \le \mu\}= \bigcup_{j=1}^\ell I_j
\]
with disjoint (non-empty) intervals $I_j$.  
Let $(\chi_k)_{1\leq k\leq
\ell}$ be a family of $C^\infty_c(\R,\R_+)$ functions with disjoint supports such that $\chi_k=1$ on $I_k$ and fix a small $\epsilon>0$.  
Let  $(\hbar_m)_{m\in \N}$ be a sequence of semiclassical parameters
such that for some $\alpha>0$, $\hbar_m \leq m^{-\alpha}$ if $m$ is large enough, and define
\[
\Pi_{m}=\1_{(-\infty,\mu]}\bigg(-\hbar^2_m\Delta+V+\sum_{j=1}^{\ell} w_j\chi_j\bigg), \qquad w\in [-\epsilon,\epsilon]^\ell.
\]
Then, for almost every $w\in [-\epsilon,\epsilon]^\ell$, for any
$f\in C^{\infty}_c(\R,\R)$, it holds  uniformly for $\eta\in\C$ sufficiently small, 
\[
\log\det(I+(e^{\eta f}-1)\Pi_{m})=\eta\tr(f\Pi_{m})+\frac{\eta^2}{2}\sum_{j=1}^{\ell} \Sigma_{(W_j,\mu)}^2(f) + \underset{m\to\infty}{o(1)} 
\]
where $W_j = V+w_j$ on a neighbourhood of $I_j$, is strictly larger
than $\mu$ outside, and $\Sigma_{(W_j,\mu)}^2$ is as in Theorem~\ref{thr:single_well}. 
\end{theorem}

\medskip

This behaviour is in sharp contrast with that of multi-cut random matrix models. 
Indeed, it was pointed out by Pastur \cite{Pastur_06} that because
orthonormal polynomials are delocalized on every cut in the
semi-classical regime, the numbers of eigenvalues in the different
cuts of the equilibrium measure has non-trivial order 1 fluctuations
for large $N$.  Then, for real-analytic potentials, it was established
in \cite{shcherbina_fluctuations_2013} and \cite[Theorem
1.6]{borot_asymptotic_2013-2} that these fluctuations are given by an
explicite multivariate discrete Gaussian law. This phenomena entails that
the fluctuations of a generic linear statistic are also uniformly
bounded in $N$, but they are generically non-Gaussian. 
For unitary-invariant Hermitian random matrices ($\beta=2$), this
problem was revisited in \cite{charlier_asymptotics_2021} using
Riemann-Hilbert asymptotics for orthogonal polynomials in the
multi-cut regime. In particular, the connection with the Gaussian free
field on a Riemann surface is discussed in detail in
\cite[Section~1.5]{charlier_asymptotics_2021}. In our situation, the
eigenfunctions of $-\hbar^2\Delta+V+\sum w_j\chi_j$ are, generically,
localised on exactly one of the intervals $I_j$, which we use to prove
that the fluctuations
are Gaussian.

It is not clear what happens in the resonant case, such as for the
symmetric double well potential $V(x)=x^4-x^2$ at energies
$\mu<0$. From the results of \cite{helffer_puits_1985}, eigenvalues
appear in exponentially close pairs and if $\mu$ is not inbetween
them (so that $N$ is even), then again the projector exponentially close to the sum of the
two one-well projectors. In contrast, if $N$ is odd, we do not know
whether a central limit theorem should hold.

\paragraph{Proof strategy and additional comments.} 
Our proof of Theorem~\ref{thr:single_well} relies crucially on the \emph{complete integrability} of one-dimensional Schrödinger operators. 
Namely, under the assumptions of Theorem~\ref{thr:single_well}, the eigenfunctions of the operator $-\hbar^2\Delta+V$ with eigenvalues in a neighborhood of $\mu$ have all order asymptotic expansions. This is usually performed using Wentzel--Kramers--Brillouin (WKB) approximations in quantum mechanics \cite{colin_de_verdiere_spectre_1980}. 
Rather than using directly these approximations, our method exploits
the \emph{quantum action-angle theorem} (Proposition~\ref{prop:aa}) to
reduce (by an approximate conjugation of the Fredholm determinant
\eqref{Laplace}) the problem to the case of the harmonic oscillator
(Example~\ref{ex:gue}), at the price of working with a general
\emph{pseudo-differential operator} instead of a multiplication
operator.
In this context, we obtain Szeg\H{o}-type asymptotics (Proposition~\ref{prop:Szegoasymp}) for certain Fredholm determinants of  pseudodifferential operators.

Our proof involves the fact that pseudodifferential operators have an \emph{approximate Toeplitz}
structure in the eigenbasis of the harmonic oscillator. 
This reduces the problem to Szeg\H{o}-type asymptotics for which one can apply classical methods such as the cumulant method and the Dyson-Hunt-Kac (DHK) combinatorial Lemma, see \cite{Kac54} and  \cite[Chap.~6.5]{simon_orthogonal_2005}.
This method has already been successfully used in another generalisation of the strong Szeg\H{o} limit theorem using microlocal tools \cite{guillemin_spectral_1997}. 

When using the DHK combinatorial approach, there is no need for a
symmetry hypothesis (either the self-adjointness used in \cite{Kac54}
or the time-reversibility used in \cite{guillemin_spectral_1997}) to
obtain a strong Szeg\H{o} theorem; we clarify this by reviewing this
approach.

Szeg\H{o}-type limit theorems also hold for general one-cut unitary-invariant Hermitian random matrix ensembles, we refer to \cite{breuer_central_2017,lambert_clt_2018} for different approaches to Szeg\H{o}-type asymptotics using the connection with orthogonal polynomials. In particular, different forms of the DHK formula are discussed in \cite[App~A]{lambert_clt_2018}. 

The action-angle strategy requires the non-degeneracy ``Airy-type
edge'' generic condition $\nabla V\neq 0$ on $\{V=\mu\}$. It would
be interesting to consider the case of hyperbolic critical points (such as
appearing in Landau-type potentials), where
the asymptotic behaviour of eigenfunctions is also known \cite{colin_de_verdiere_equilibre_1994}.

The proof of Theorem~\ref{thr:multiple_wells} relies on well-known
localisation techniques (see for instance \cite{helffer_multiple_1984,helffer_multiple_1985}) to show that, generically (in the sense of Theorem \ref{thr:multiple_wells}), the projector $\Pi_{\hbar}$ can be approximated by a commuting sum of projectors which correspond to the potential localized on each well. 
Hence, in this case, we can reduce the asymptotics of linear statistics to the case of Theorem~\ref{thr:single_well}.

\paragraph{Organisation of the paper.}
The proof of Theorem~\ref{thr:single_well} (and its generalisation
to pseudo-differential operators, Proposition~\ref{prop:Szegoasymp}) is given in Section~\ref{sec:one-well-case} and is organized as follows.
In Sections~\ref{sec:notations} and~\ref{sec:semiclassical-tools} we review the classical action-angle theorem for certain one-dimensional Hamilton flow and its quantum analogue for Schr\"odinger operators. 
This also gives us the opportunity to introduce several important
notation and basic results from semiclassical analysis.

In Section~\ref{sec:Toeplitz}, we establish that as a consequence of the \emph{quantum angle-action theorem} and the asymptotic properties of the Hermite functions, a certain class of pseudo-differential operators have an approximately Toeplitz structure (as $\hbar\to0$) in the eigenbasis of $-\hbar^2\Delta+V$. 
In Section~\ref{sec:strong-szegho-limit} we gather some results concerning the Taylor expansion of Fredholm determinants and its continuity property with respect to perturbations. Finally, in Section~\ref{sec:end-proof}, we combine these results to complete the proof of Theorem~\ref{thr:single_well}.
The core of the argument relies on the approximately Toeplitz property of pseudo-differential operators and a combinatorial approach to the Strong Szeg\H{o} theorem known as the Dyson-Kac-Hunt (DHK) formula. 

The proof of Theorem~\ref{thr:multiple_wells} is given in Section~ \ref{sec:several-wells}. It relies on Theorem~\ref{thr:single_well}, localisation techniques to isolate the contribution from each well and the genericity condition to guarantee that there are no \emph{resonances} between the wells. 

\section{One-cut case -- Proof of Theorem~\ref{thr:single_well}}
\label{sec:one-well-case}

Throughout this section,  we assume that the potential  $V\in C^{\infty}(\R,\R)$ is such that $V\to +\infty$ when
$|x|\to +\infty$ and we fix $\mu\in \R$ such that $\{V\leq \lambda\} =
[x_0^-(\lambda),x_0^+(\lambda)]$ with $x_0^-(\lambda)<x_0^+(\lambda)$ for $\lambda$ in a
neighborhood of $\mu$. Furthermore, we assume a non-degeneracy
condition $V'(x_0^{\pm}(\lambda))\neq 0$ for every $\lambda$ in this
neighbourhood of $\mu$. Under these hypotheses,  the maps $\lambda
\mapsto x_0^\pm(\lambda)$ are $C^{\infty}$-diffeomorphisms. Later on
we will assume more technical assumptions about the behaviour of $V$
near infinity, but they can be lifted (see \cite[Section 2.2]{deleporte_universality_2021}).

\subsection{Classical action-angle theorem}
\label{sec:notations}

Let $(\phi_t)_{t\in\R}$ denotes the Hamilton flow of $\mathcal{H}
:(x,\xi)\mapsto \xi^2+V(x)$, that is, for every $(x_0,\xi_0)\in \R^2$,
the function $\R^2 \ni(x_0,\xi_0) \mapsto
(x_t,\xi_t)=\phi_t(x_0,\xi_0)\in \R^2$ solves the evolution equation
\begin{equation} \label{Hflow}
\begin{cases}
\partial_t x_t  = 2 \xi_t \\
\partial_t \xi_t  =  - V'(x_t).
\end{cases}
\end{equation}

A remarkable property of one-dimensional Hamiltonian dynamics is
\emph{complete integrability}. The dynamics preserve the energy
$\mathcal{H} $, and the flow is periodic along the curves
\begin{equation} \label{scurve}
\mathcal{C}_{\lambda} :=\big\{(x,\xi)\in \R^2, \mathcal{H}(x,\xi)=\lambda \big\} ,
\end{equation}
with (shortest) period $T(\lambda)$. 
By assumptions, these curves are simply connected and $T(\lambda)<\infty$ for $\lambda$ in a neighbourhood of $\mu$.

The flow $\phi_t$ is, on a neighbourhood of $\mathcal{C}_{\mu}$,
conjugated to the Hamiltonian flow of a function of the harmonic
oscillator; letting
\begin{equation} \label{g}
\g(\lambda):=\frac{1}{2\pi}{\rm Area}\big(\{(x,\xi)\in
\R^2,\mathcal{H}(x,\xi)\leq\lambda\}\big),
\end{equation}
then the formula
\[
\kappa:(\phi_t(x_0^+(\lambda),0))\mapsto
\sqrt{2\g(\lambda)}\Big (\cos\tfrac{2\pi t}{T(\lambda)},\, -\sin\tfrac{2\pi t}{g(\lambda)} \Big)
\]
defines a $C^{\infty}$-diffeomorphism from a $\phi_t$-invariant
neighbourhood of $\mathcal{C}_{\mu}$ to an annular neighbourhood of
$\{x^2+\xi^2=2\g(\mu)\}$. In fact, $\kappa$ is area-preserving and
$\g\circ \mathcal{H}=\frac{|\kappa|^2}{2}$.

Given $a\in C^{\infty}(\R^2,\R)$ and $I$ in a neighbourhood of $\g(\mu)$,
we can now define the Fourier coefficients of $a$ along the Hamilton
flow as
\begin{equation}\label{Fouriera}
\hat{a}_j(I):=\frac{1}{T(\lambda)}\int_0^{T(\lambda)}e^{-\frac{2i\pi jt}{T(\lambda)}}a(\phi_t(x_0^+(\lambda),0))\dd  t  \qquad \qquad \lambda=\g^{-1}(I) , \quad k\in\Z. 
\end{equation}
An equivalent definition is
\begin{equation} \label{Fourierb}
\hat{a}_j(I):=\frac{1}{2\pi}\int_0^{2\pi}a\circ\kappa^{-1}(\sqrt{2I}\cos(\theta),\sqrt{2I}\sin(\theta))e^{ij\theta}\dd\theta, \qquad  k\in\Z.
\end{equation}

Then, in the context of Theorem~\ref{thr:single_well}, the asymptotic variance $\Sigma$ is defined by, for any  $f\in C^{\infty}_c(\R,\R)$, 
\begin{equation}\label{eq:Sigma}
\Sigma_{(V,\mu)}^2(f) := \sum_{k\in\N} k \big|\hat{f}_k(
\g(\mu))\big|^2 ,
\end{equation}
where we view $f$ as a (smooth) function on $\R^2$ independent of the second
variable. This amounts to the $H^{\frac 12}$ seminorm of the periodic
function $f\circ \psi$, where
\[
\psi(\theta)=\pi_1\circ\phi_{\frac{\theta T(\mu)}{2\pi}}(x_0^+(\mu),0) , \qquad \theta\in[0,2\pi]
\]
and $\pi_1:\R^2\to \R$ is the projection on the first coordinate (position). 

\subsection{Semiclassical tools}
\label{sec:semiclassical-tools}

Here we collect some useful estimates for the eigenfunctions of
one-dimensional Schr\"odinger operators under the hypotheses above. 
Recall the definition of the projection $\Pi_\hbar$, \eqref{proj}, and
the spectral curve $\mathcal{C}_\mu$, \eqref{scurve}.

We will use several techniques from semiclassical analysis, notably a
quantum version of the considerations of Subsection
\ref{sec:notations}. We begin with a quick overview of
pseudodifferential operators and their application to the spectral
theory of Schrödinger operators, referring to
\cite{dimassi_spectral_1999,zworski_semiclassical_2012} for more detailed introductions.

\begin{defn}\label{def:pseudos}
Let $a\in \mathcal{S}(\R^2,\C)$ and $\hbar>0$. The Weyl quantization
of $a$ is the operator with kernel
\[
Op_{\hbar}(a)(x,y)=\frac{1}{2\pi\hbar}\int
e^{i\frac{\xi(x-y)}{\hbar}}a(\tfrac{x+y}{2},\xi)\dd \xi.
\]
This is a trace class operator on $L^2(\R)$, whose trace norm is controlled by a
finite number of seminorms of $a$ in $\mathcal{S}$. Moreover, its
operator norm satisfies the G\aa rding inequality
(\cite{zworski_semiclassical_2012}, Theorem 4.32)
\begin{equation} \label{Garding}
\|Op_{\hbar}(a)\|_{L^2\to L^2}=\|a\|_{L^{\infty}}+\O(\hbar).
\end{equation}
\end{defn}
More general function classes are well-adapted to Weyl quantization,
such as the symbol class $S^0(\R^2,\C)$; an element of $S^0$ is a function $a\in C^{\infty}(\R^2,\C)$ such that
\[
\forall (j,k)\in\N^2,\, \exists C_{j,k}\in\R_+,\quad
\sup_{(x,\xi)\in \R^2}  |\partial^j_x\partial^k_{\xi}a(x,\xi)|\leq C_{j,k}.
\]
Elements of $S^0$ also lead to bounded operators satisfying \eqref{Garding}.

\smallskip

In this framework, we will prove the following generalization of Theorem~\ref{thr:single_well}.

\begin{prop}[Szeg\H{o} asymptotics]\label{prop:Szegoasymp}
Let $\mu>0$ and $\Pi_{\hbar}= \Pi_{\hbar}(\mu)$ be given by \eqref{proj} under the assumptions of Theorem~\ref{thr:single_well}. 
Let $a_\hbar \in S^0(\R^2,\C)$ be a symbol with $\|a_\hbar\|_{L^\infty} \le c <1$. Then, 
\[
\log\det(\mathrm{I}+Op_{\hbar}(a)\Pi_{\hbar})
= \tr(\log(1+Op_{\hbar}(a))\Pi_\hbar)
+ \frac{1}2  \sum_{\ell\in\Z} |\ell| \widehat{f}_\ell  \widehat{f}_{-\ell}+ \underset{\hbar\to0}{o(1)} 
\]
where $f = -\log(1+a_0) =  \sum_{r\ge 1} \tfrac{(-1)^r}{r} a_0^r$ converges to $f\in C(\R^2,\C)$ since $\|a_0\|_{L^\infty} \le c$ and the Fourier coefficients $(\widehat{f}_\ell )_{\ell\in\Z}$ are given by formula \eqref{Fouriera} with $I =\g(\mu)$. 
\end{prop}

Observe that the function  $a:(x,\xi)\mapsto e^{\eta f(x)}-1$ belongs
to $S^0(\R^2,\C)$ if $f\in C^{\infty}_c(\R,\R)$ with
$\|a\|_{L^{\infty}}<1$ if the parameter $\eta\in\C$ is small enough;
moreover it acts on $u\in L^2(\R)$ by multiplication: $Op_{\hbar}(a)u:x\mapsto e^{\eta f(x)}u(x)$.
Then, by Proposition~\ref{prop:Szegoasymp} (rescaling $f$ to $\eta f$), we conclude that as $\hbar\to0$, 
\[
\log \det(1+a\Pi_\hbar) = \eta\tr(f\Pi_{\hbar}) +   \frac{\eta^2}2 \sum_{\ell\in\Z} |\ell| \widehat{f}_\ell  \widehat{f}_{-\ell} + o(1) .
\]
Hence, by definition of the variance $\Sigma_{(V,\mu)}^2$, \eqref{eq:Sigma},  this proves Theorem~\ref{thr:single_well}  in the one-cut case. 

\smallskip

Pseudodifferential operators are well-adapted to the study of
Schrödinger operators under some conditions on $V$ at
infinity\footnote{By elementary (Agmon-type) techniques, changing $V$
away from $\{V\leq \mu\}$ has negligible impact on the quantities
studied here; see \cite[Section 2.2]{deleporte_universality_2021} for
details.}. In the rest of this article, we will always assume that there exists $m>0$ such that
\[
\forall k\in \N, \exists C_k, \forall x\in \R,|\partial^kV(x)|\leq C_k(1+|x|)^{m}
\]
and
\[
\exists c_0, R >0 , \, \forall |x|>R,\, |V(x)|\geq c_0(1+|x|)^m.
\]
Under these conditions, not only is the operator $-\hbar^2\Delta+V$ bounded
from below with compact resolvent, but smooth compactly supported
spectral functions of $-\hbar^2\Delta+V$ are given by
pseudodifferential operators with compactly supported symbol
(\cite{zworski_semiclassical_2012}, Theorem 14.9).
Moreover, the next Lemma implies that we can also assume that the symbol $a_\hbar$ in Proposition~\ref{prop:Szegoasymp} is supported on a small neighborhood of the curve~$\mathcal{C}_\mu$,~\eqref{scurve}. 
\begin{lem}[Micorlocalisation] \label{lem:microloc}
Let $(\phi_j)_{j\in \N_0}$ be a Hilbert basis of eigenfunctions of $-\hbar^2\Delta+V$ with eigenvalues $(\lambda_j)_{j\in\N_0}$.
For every $a\in S^0(\R^2,\C)$ such that $\supp(a)\cap \mathcal{C}_\mu = \emptyset$, there exists a (small) $\delta>0$ so that for every $j\in\N_0$, 
\[
Op_{\hbar}(a) \phi_j = \O_{L^2}(\hbar^{\infty}) \qquad\text{uniformly for } |\lambda_j -\mu| \le \delta.
\]
Then, as $\hbar\to0$, 
\[
\big\| [\Pi_{\hbar},Op_{\hbar}(a)]\big\|_{J^1}=\O(\hbar^{\infty}).
\]
\end{lem}

\begin{proof}
The first claim is exactly \cite[Theorem 6.4]{zworski_semiclassical_2012}. To obtain the second claim, we decompose 
\begin{equation} \label{decomp}
\Pi_{\hbar} =\1_{(-\infty,\mu]}(-\hbar^2\Delta+V)
=  \vartheta(-\hbar^2\Delta+V)  + \chi \1_{(-\infty,\mu]}(-\hbar^2\Delta+V) 
\end{equation}
where $\chi\in C^\infty_c(\R,[0,1])$  is supported on a neighborhood
of $\mu$, and $\vartheta\in C^\infty_c(\R,[0,1])$ is
supported on $\{\xi^2+V(x)<\mu\}$. By the functional calculus
(\cite{zworski_semiclassical_2012}, Theorem 14.9), one has
\[
\vartheta(-\hbar^2\Delta+V) = Op_{\hbar}(b) + \O_{J^1}(\hbar^\infty)
\]
where $b_\hbar\in \mathcal{S}$ satisfies $b_\hbar=b_{\hbar}^2$ on  $\supp(a)$.
Then, using  the composition formula for Weyl quantization (e.g.~\cite{zworski_semiclassical_2012}, Theorem 4.11 and 4.12), 
\[
[Op_{\hbar}(b), Op_{\hbar}(a)] = Op_{\hbar}(b\#a-a\#b) 
\qquad\text{where}\quad
b\#a = a\#b + \O_{\mathscr{S}}(\hbar^\infty) . 
\]
Consequently, by \cite[Theorem 4.21]{zworski_semiclassical_2012}, $\|
Op_{\hbar}(b\#a-a\#b)  \|_{L^2\to L^2} =  \O(\hbar^\infty) $. Using
the Weyl-type bound
${\rm rank}(\vartheta(-\hbar^2\Delta+V))=\O(\hbar^{-1})$, we finally obtain
\begin{equation} \label{est1}
\| [ \vartheta(-\hbar^2\Delta+V),  Op_{\hbar}(a)] \|_{J^1} =  \O(\hbar^\infty)  .
\end{equation}
Moreover, we also have 
\[
\| \chi(-\hbar^2\Delta+V)Op_{\hbar}(a) \|_{J^2}
=  \sum_{j\in\N_0} \chi(\lambda_j)^2 \big\| Op_{\hbar}(a) \phi_j \big\|^2  
\]
where every term  is $\O(\hbar^{\infty})$ (by the first claim as $\chi$ is supported in a $\delta$-neighborhood of $\mu$). Since there are $\O(\hbar^{-1})$ non-zero terms, this shows that 
\begin{equation} \label{est2}
\| \chi(-\hbar^2\Delta+V)Op_{\hbar}(a) \|_{J^2} = \O(\hbar^{\infty})  .
\end{equation}
Since both operators on the RHS of \eqref{decomp} have rank $\O(\hbar^{-1})$, the claim follows by combining the estimates \eqref{est1} and \eqref{est2}.
\end{proof}

Like the classical Hamiltonian dynamics in two-dimensional phase space, the quantum evolution problem and the eigenvalue problem are \emph{integrable} for one-dimensional pseudo-differential operators. This basic
structure has been thoroughly exploited, notably through 
WKB expansions for eigenfunctions, and it lies at the foundation of quantum mechanics. The modern formulation
and generalisation of this structure is the
existence of a unitary conjugation of a Schrödinger operator into a function of the harmonic oscillator. This
conjugation is \emph{semi-global}, that is, holds in the vicinity of a given
energy level $\mu$, which is either regular (no critical point) or elliptic (critical points are
Morse and the Hessian has positive determinant).

\begin{prop}[Quantum action-angle theorem] \label{prop:aa}
Recall that $\g$ is the area map given by \eqref{g}. 
There exists a bounded operator $U_{\hbar}$ on $L^2(\R)$ such that for every $\chi \in C^{\infty}_c(\R,\R)$ supported in a small neighborhood of $\mu$,  
\[
U_{\hbar} \chi(H_{\hbar})U_{\hbar}^*= \vartheta_\hbar\big(-\tfrac{\hbar^2}{2}\Delta+\tfrac{x^2}{2}\big)+\O_{J^1}(\hbar^{\infty}),
\]
where the symbol $\vartheta_\hbar \in C^{\infty}_c(\R,\R)$ admits a formal
expansion in powers of $\hbar$, with principal part $\vartheta_0 = \chi \circ
\g^{-1}$, and every terms is supported  in a small neighbourhood
of the support of $\vartheta_0$.
Moreover, $U_\hbar$ is microlocally unitary near $C_{\mu}$:
\begin{itemize}
\item for every $a\in C^{\infty}_c(\R^2,\R)$ supported in a small neighborhood of the curve $\mathcal{C}_\mu$,
\[
Op_{\hbar}(a)U_{\hbar}^*U_{\hbar}=Op_{\hbar}(a)+\O_{J^1}(\hbar^{\infty}) .
\]
\item for every  $b\in C^{\infty}_c(\R^2,\R)$ in a small neighborhood of the circle $\kappa(\mathcal{C}_\mu)$,
\[
Op_{\hbar}(b)U_{\hbar}U_{\hbar}^*=Op_{\hbar}(b)+\O_{J^1}(\hbar^{\infty}).
\]
\item  for every $a\in C^{\infty}_c(\R^2,\R)$ supported in a small neighborhood of the curve $\mathcal{C}_\mu$, there exists $b_{\hbar}$ with compact support and a formal expansion
in powers of $\hbar$ with first term
\(
b_0=a\circ \kappa^{-1}
\)
such that
\[
U_\hbar Op_{\hbar}(a)U_\hbar^*=Op_{\hbar}(b)+\O_{J^1}(\hbar^{\infty}).
\]
\end{itemize}
\end{prop}
A good presentation (in French) of the quantum action-angle theorem near regular
trajectories can be found in \cite{vu_ngoc_systemes_2006}; the statement above
corresponds to Théorème 4.1.8 (see also \cite{vu_ngoc_formes_2000,colin_de_verdiere_spectre_1980}). There are several versions of this theorem;
usually one conjugates the operator into a function of
$-i\hbar\partial/\partial \theta$ acting on $L^2(S^1)$, but in our
context, it is simpler to conjugate to the harmonic oscillator,
because of the absence of eigenvalue shifts (vanishing of the Maslov
index), as will come into play below.

We begin with a description of the model case.
\begin{prop}[\cite{zworski_semiclassical_2012}, Chap.~6.1]\label{prop:Hermite}
The operator $H_\hbar =-\tfrac{\hbar^2}{2}\Delta+\tfrac{x^2}{2}$ is
the (quantum) harmonic oscillator. Its eigenfunctions $(\psi_j)_{j\in
\N_0}$ are the (semi-classical) Hermite functions. In particular, we have $H_\hbar  = A_\hbar^*A_\hbar + \hbar/2$ where $A_\hbar^* =  \tfrac{1}{\sqrt{-2}}(\hbar \partial_x + x)$ is the creation operator. 
We have $\psi_0(x) \propto e^{-x^2/2\hbar} $ for $x\in\R$ and $\psi_j \propto A_\hbar^{*j} \psi_0 $ for $j\in\N$ so that $H_\hbar \psi_j = (j+1/2) \hbar \psi_j$ for $j\in \N_0$. 
\end{prop}

\begin{rem}\label{rk:eigsep}
Let $\Omega$ be a small (but independent of $\hbar$) neighbourhood
of $\mu$.
Propositions \ref{prop:aa} and \ref{prop:Hermite} imply that for any $\chi\in C^{\infty}_c(\Omega,\R)$, the spectrum of $\chi(H_{\hbar})$  lies
$\O(\hbar^{\infty})$-close to $\{\vartheta_{\hbar}((n+\tfrac
12)\hbar),n\in \N\}$, see for instance \cite[Theorem 8.20]{helffer_spectral_2013}.
This implies that the eigenvalues of $H_{\hbar}$ in $\Omega$
are simple and separated from each other by about $\hbar$ (one can see this by
considering the spectra of $\chi(H_{\hbar})$ for two different
functions $\chi$).
\end{rem}

In our context, using the conventions from Proposition~\ref{prop:Hermite}, the relevant consequence of Proposition \ref{prop:aa} is the following result.

\begin{prop}\label{prop:angle-action-applied} 
There exists a Hilbert basis $(\phi_j)_{j\in \N}$ of eigenfunctions of $-\hbar^2\Delta+V$ such that the following holds.
For every $a\in C^{\infty}_c(\R^2,\R)$ supported in a small neighborhood of the curve $\mathcal{C}_\mu$, there exists a symbol\footnote{with formal expansion in powers
of $\hbar$ and supported in a small neighborhood of $a\circ \kappa^{-1}$} $b_{\hbar}\in C^{\infty}_c(\R^2,\R)$ such that
\(
b_\hbar=a\circ \kappa^{-1}+\O(\hbar),
\)
and, uniformly for $j, k\in \N_0$, 
\begin{equation}\label{eq:aa-applied}
\langle \phi_j,Op_{\hbar}(a)\phi_k\rangle=\langle
\psi_{j},Op_{\hbar}(b)\psi_{k}\rangle + \O(\hbar^{\infty}) . 
\end{equation}

In particular, ${\rm rank}(\1_{(-\infty,\mu]}(-\hbar^2\Delta+V))=\hbar^{-1}\g(\mu)+\O(1)$.
\end{prop}

\begin{proof}
Since the eigenvalues of $H_{\hbar}$ lying in a neighborhood of $\mu$ are separated by  about $\hbar$ (Remark~\ref{rk:eigsep}), the operator $U_{\hbar}$ from Proposition \ref{prop:aa} maps the eigenfunction of $H_{\hbar}$ onto the corresponding Hermite function\footnote{\label{fn:stability}The proof of this
well-known claim goes as follows: suppose a bounded self-adjoint operator
$H_0$ has an isolated simple eigenvalue $\lambda_0$ with
$\dist(\lambda_0,\sigma(H_0)\setminus \lambda_0)=\epsilon_0>0$. If 
$\psi$ is normalised in $L^2$ and approximately solves the eigenvalue equation:
$\|(H_0-\lambda_0)\psi\|=\delta\ll \epsilon_0$, then 
$\psi=\langle \psi_0,\psi\rangle \psi_0+\Pi\psi$ where $\Pi$ is the
spectral projector onto the orthogonal of $\psi_0$ and
$\|(H_0-\lambda_0)\Pi\psi\|=\delta$. Hence, since $(H_0-\lambda_0)$ is
invertible on the range of $\Pi$, with inverse bounded by
$\epsilon_0^{-1}$, we conclude that $\|\Pi\psi\|\leq \delta/\epsilon_0$.}; for every
$j\in \N$ such that $\lambda_j\in [\mu-\epsilon,\mu+\epsilon]$ there exists
$k =k(j)\in \N$ and a phase $\alpha_{k}$ such that
\begin{equation}\label{eq:Uh_on_eigen}
\phi_j=e^{i\alpha_k}U_{\hbar}\psi_k+\O(\hbar^{\infty}).
\end{equation}
In the last equation, one has in fact $j=k$. This follows by computing the Maslov index of $U_{\hbar}$ (see \cite{colin_de_verdiere_spectre_1980,vu_ngoc_bohr-sommerfeld_2000}).
A direct way to see
that indices match is to use a deformation argument and a stronger
action-angle formula, which we now explain.

Let $p:\R^2\to \R$ be such that $p(x,\xi)=V(x)+\xi^2$ near
$C_{\mu}$ and such that, on $\{p<2\mu\}$, there is only one
critical point which is a non-degenerate local minimum. In this
situation, we can extend the action-angle theorem to the whole
bottom of the spectrum (as in \cite{vu_ngoc_systemes_2006},
Theorem 4.22), and since the bottom of the spectra must match, for
the eigenfunctions $\tilde{\phi}_j$ of $Op_{\hbar}(p)$, one has
exactly $U_{\hbar}\tilde{\phi}_j=e^{i\alpha_j}\psi_j +\O(\hbar^{\infty})$
for all $j$ with $\tilde{\lambda}_j<2\mu$.

Now, by microlocalisation arguments (Lemma~\ref{lem:microloc}), the eigenvalues
and eigenfunctions of $Op_{\hbar}(p)$ and $-\hbar^2\Delta+V$ are
identical, up to $\O(\hbar^{\infty})$, for all eigenvalues in $[\mu-\epsilon,\mu+\epsilon]$ and there is no index shift in this region because one can construct a
homotopy between these two operators which fixes the number of eigenvalues in $\{\lambda<\mu-\epsilon\}$ and $\{\lambda\le \mu+\epsilon\}$. Thus, if
$\lambda_j\in [\mu-\epsilon,\mu+\epsilon]$,
\[
\tilde{\phi}_j=\phi_j+\O(\hbar^{\infty})
\]
without index shift, which allows us to conclude.

To conclude the proof of Proposition \ref{prop:angle-action-applied},
we can replace $\phi_k$ by $e^{-i\alpha_k}\phi_k$ and apply Proposition \ref{prop:aa} if 
both $\lambda_j, \lambda_k\in [\mu-\epsilon,\mu+\epsilon]$, we obtain
\[
\langle \phi_j,Op_{\hbar}(a)\phi_k\rangle=\langle
\psi_{j},Op_{\hbar}(b)\psi_{k}\rangle + \O(\hbar^{\infty}) . 
\]
Otherwise,
by Proposition \ref{lem:microloc}, say $\phi_j$ is microlocalised away
from the support of $a$, while $\psi_j$ is microlocalised away from
the support of $a\circ \kappa^{-1}$; thus both sides of
\eqref{eq:aa-applied} are $\O(\hbar^{\infty})$; which concludes the
proof.
\end{proof}

\subsection{Approximate Toeplitz structure} \label{sec:Toeplitz}

It is possible to derive an asymptotic formula for $
\langle \psi_j,Op_{\hbar}(b),\psi_k\rangle
$ using either the induction formula on the Hermite functions $(\psi_k)_{k\in\N_0}$ or their
WKB asymptotics as $\hbar\to0$. Similar asymptotic formulas also appear in the
literature \cite{guillemin_spectral_1997} where the proof is quite
involved, notably using the special form of the Guillemin-Wodzicki
residue in the case of a periodic classical flow. We use a different and simpler approach in this article; a unitary
conjugation to Berezin--Toeplitz quantization on the Bargmann--Fock space.

\begin{prop}[\cite{folland_harmonic_1989}, Propositions 2.96 and 2.97; see also
\cite{zworski_semiclassical_2012}, Chapter 13]\label{prop:conj_Barg}
Recall the notation from Proposition~\ref{prop:Hermite}.
Let $\mathcal{B}_{\hbar}$ be the Bargmann transform (or
wavepacket transform) defined by the integral
kernel
\[
\C\times \R\ni (z,x)\mapsto
\mathcal{B}_{\hbar}(z,x)=\frac{1}{(\pi\hbar)^{1/4}}\exp\left(-\frac{z^2+|z|^2-2\sqrt{2}zx+x^2}{2\hbar}\right).
\]
Then $\mathcal{B}_{\hbar}$ is a unitary transformation from $L^2(\R)$ into
\[
\mathcal{H}_{\hbar}=\{f\in L^2(\C), z\mapsto e^{\frac{|z|^2}{2\hbar}}f(z)\text{
is holomorphic}\}
\]
such that for all $j\in \N_0$,
\[\mathcal{B}_{\hbar}\psi_j:z\mapsto
\frac{\hbar^{-\frac{1+j}{2}}}{\sqrt{\pi j!}}z^je^{-\frac{|z|^2}{2\hbar}} .\]

Then, for every $b\in C^{\infty}_c(\R^2,\R)$, there exists
$f_\hbar\in C^{\infty}_c(\C,\R)$ with a formal expansion in powers of
$\hbar$ and main term $f_0(x+i \xi) = b_0(x,\xi) $ such that, 
uniformly as $\hbar\to 0$, for all $u,v\in \mathcal{H}_{\hbar}$ with $\|u\|=\|v\| =1$, 
\begin{equation}\label{eq:Berezin-Toeplitz}
\langle u,\mathcal{B}_{\hbar}Op_{\hbar}(b)\mathcal{B}_{\hbar}^*v\rangle=\int
u(z)\overline{v(z)} f_\hbar(\sqrt{2}z)\dd z +\O(\hbar^{\infty}).
\end{equation}
\end{prop}

\begin{rem}\label{rem:Hermite}
$\mathcal{B}_{\hbar}A_\hbar^* \mathcal{B}_{\hbar}^*  : u\mapsto zu $ corresponds to the usual creation operator on $\mathcal{H}_{\hbar}$. In particular, formula \eqref{eq:Berezin-Toeplitz} is exact in this case and the normalization of $\mathcal{B}_{\hbar}$ is consistent with the physics literature.
\end{rem}

With this conjugation, one can easily obtain the asymptotics of matrix
elements for $Op_{\hbar}(b)$ in the Hermite basis for any $b\in C^{\infty}_c$.

\begin{prop}\label{prop:limit_Toep_Bargmann}
Let $b_\hbar\in C^{\infty}_c(\R^2,\R)$.
For every $\epsilon>0$, as $\hbar \to 0$, uniformly for $j, k\geq \epsilon\hbar^{-1}$,
\[
\langle \psi_j,Op_{\hbar}(b)\psi_k\rangle = \hat{b}_{k-j}(\tfrac{j+k}{2}\hbar)+\O_\epsilon(\hbar |j-k|^{-\infty})+\O(\hbar^{\infty}).
\]
where, for $\lambda>0$,  $\big(\hat{b}_k(\lambda)\big)_{k\in\Z}$ denotes the Fourier coefficients of the function 
$\theta\in[0,2\pi] \mapsto f_0(\sqrt{2\lambda} e^{\i\theta} )$  with  $f_0(x+i \xi) = b_0(x,\xi) $. 
\end{prop}

\begin{proof}
We have for $j,k\in\N_0$,
\[
\langle \psi_j,Op_{\hbar}(b)\psi_k\rangle
= \int_\C u_j(z)\overline{u_k(z)} f_\hbar(\sqrt{2}z)\dd z +\O(\hbar^{\infty})
\]
where $u_j = \mathcal{B}_{\hbar}\psi_j$ and $ f_\hbar$ are as in Proposition~\ref{prop:conj_Barg}.
In particular, $u_k(z)= \gamma_k z^k u_0(z)$ with $\gamma_k =  \sqrt{\frac{\hbar^{-k}}{k!}}$ for $k\in\N_0$, so that
\begin{equation} \label{idu}
\int_\C u_j(z)\overline{u_k(z)} f_\hbar(\sqrt{2}z) \dd z 
= \frac{\gamma_j\gamma_k}{\hbar}\int_0^\infty \hat{f}_{k-j}(\lambda) e^{-\lambda/\hbar} \lambda^{\frac{j+k}{2}}\dd\lambda
\end{equation}
where $\big(\hat{f}_k(\lambda)\big)_{k\in\Z}$ denotes the Fourier coefficients of the function 
$\theta\in[0,2\pi] \mapsto f_\hbar(\sqrt{2\lambda} e^{\i\theta} )$ for $\lambda>0$.

Note that for $\alpha>0$, the function $\lambda \mapsto \lambda - \alpha \log \lambda$ is strictly convex on $\R_+$ with a unique minimum $\varpi(\alpha)$ for $\lambda=\alpha$. Let $\chi \in C^\infty_c(\R)$ be a smooth cut-off which equals to 1 on a neighborhood of $\alpha$.
By the Laplace method, since $f \in C^2(\R_+)$ is bounded, it holds as $\hbar \to0$,
\[\begin{aligned}
\int_0^\infty f(\lambda) e^{-(\lambda - \alpha \log \lambda)/\hbar} \dd\lambda  
&=  \int f(\lambda) \chi(\lambda) e^{-(\lambda - \alpha \log \lambda)/\hbar} \dd\lambda  +\O(\hbar^\infty) \\
&= 2\sqrt{\pi\hbar}\, e^{-\varpi(\alpha)/\hbar}  \big( f(\alpha) + \O(\hbar) \big)
\end{aligned}\]
where the errors are controlled by $\|f\|_{C^0}$ and $\|\chi f\|_{C^2}$ respectively. Moreover, this estimate is locally uniform for $\alpha\in\R_+$.

Moreover, by the Riemann--Lebesgue Lemma, if $f\in C^{\infty}_c(\C,\R)$, it holds for any $\epsilon>0$ and $\ell \in\N_0$,
\begin{equation} \label{RL}
\big\|  \hat{f}_{k}(\cdot) \big\|_{C^\ell(\epsilon,\infty)} = \underset{k\to\infty}{\O_{\epsilon,\ell}}\big(|k|^{-\infty}\big) .
\end{equation}

Hence, applying the Laplace method as above, we conclude that if the indices $j, k\geq \epsilon\hbar^{-1}$,
\[\begin{aligned}
\int_0^\infty \hat{f}_{k-j}(\lambda) e^{-\lambda/\hbar} \lambda^{\frac{j+k}{2}}\dd\lambda
& =  \int_0^\infty \hat{f}_{k-j}(\lambda) e^{-(\lambda - \alpha \log \lambda)/\hbar} \dd\lambda  ,\qquad \alpha= \tfrac{j+k}{2}\hbar , \\
&= 2\sqrt{\pi\hbar}\, e^{-\varpi(\alpha)/\hbar}  \Big( \hat{f}_{k-j}(\alpha) + \O_\epsilon\big(\hbar|k-j|^{-\infty}\big) \Big) .
\end{aligned}\]
In particular, we can replace $f_\hbar$ by its principal part $f_0$ while computing the Fourier coefficients, up to a similar error.

Now, going back to \eqref{idu} with $f_\hbar=1$ and $j=k$, we have for $ k\geq \epsilon\hbar^{-1}$,
\[
1= \int_\C |u_k(z)|^2 \dd z 
=2 \gamma_k^2 \sqrt{\pi/\hbar}\, e^{-\varpi(k\hbar)/\hbar} \big(1 + \O(\hbar) \big) .
\]
This implies that for $j, k\geq \epsilon\hbar^{-1}$,
\[\begin{aligned}
\int_\C u_j(z)\overline{u_k(z)} f_\hbar(\sqrt{2}z)\dd z 
&= e^{\big(\varpi(k\hbar)+\varpi(j\hbar)-2\varpi(\frac{j+k}2\hbar)\big)/2\hbar}\Big( \hat{f}_{k-j}(\tfrac{j+k}2\hbar) + \O_\epsilon\big(\hbar|k-j|^{-\infty}\big) \Big) .
\end{aligned}\]
Since $\varpi:\R_+ \ni \alpha \to \alpha(1-\log \alpha)$ is a smooth strictly concave function with 
$\| \varpi\|_{C^2(\epsilon,\infty)} \le C/\epsilon$, we conclude that 
\[
e^{\big(\varpi(k\hbar)+\varpi(j\hbar)-2\varpi(\frac{j+k}2\hbar)\big)/2\hbar}  = 1- \O_\epsilon\big(\hbar|k-j|^{2}\big)
\]
which together with \eqref{RL}, yields the final claim. 
\end{proof}

Propositions \ref{prop:angle-action-applied} and
\ref{prop:limit_Toep_Bargmann} together imply the following approximate Toeplitz formula for the matrix elements of a pseudo-differential operator.

\begin{prop}\label{prop:limit_Toep}
Recall that $(\phi_j)_{j\in \N_0}$ is a Hilbert basis of eigenfunctions of $-\hbar^2\Delta+V$.
For every symbol $a\in C^{\infty}_c(\R^2,\R)$ supported in a small neighborhood of the curve $\mathcal{C}_\mu$, it holds uniformly for $0\le j,k \le 2N(\hbar)$, 
\begin{equation*}
\langle\phi_j,Op_{\hbar}(a)\phi_k\rangle=\hat{a}_{k-j}(\tfrac{j+k}2\hbar)+\O\big(\hbar|j-k|^{-\infty}\big)
\end{equation*}
where the Fourier coefficients $\big(\hat{a}_k(I)\big)_{k\in\Z}$  are given by formula \eqref{Fourierb}.
\end{prop}

\begin{proof}
Proposition \ref{prop:angle-action-applied} (a consequence of the action-angle theorem) allows for the appropriate range of indices to compute the matrix elements in the Hermite basis after replacing the symbol $a$ by $b_\hbar$, up to a uniform error $\O(\hbar^{\infty})$.  
Then Proposition \ref{prop:limit_Toep_Bargmann} yields the asymptotics of these coefficients with the appropriate off-diagonal decay. We note that the leading terms only depend on the principal part ($b_0=a\circ \kappa^{-1}$) of the symbol $b_\hbar$ and they correspond to the Fourier coefficients  \eqref{Fourierb}  (or equivalently \eqref{Fouriera}) with $I = \hbar\tfrac{j+k}2$, while the errors $\O(\hbar^{\infty})$ are negligible compared to $\O(\hbar|j-k|^{-\infty})$ in the appropriate range. 
\end{proof}

\begin{lem}\label{prop:FM}
For every symbol $a\in S^0(\R^2,\C)$, as $\hbar\to0$, 
\[
\|[\Pi_{\hbar},Op_{\hbar}(a)]\|_{J^2} = \O(1) . 
\]
\end{lem}

\begin{proof}
Proposition
\ref{prop:limit_Toep} can be used to determine the limit of
$\|[\Pi_{\hbar},Op_{\hbar}(a)]\|_{J^2}$ as $\hbar\to 0$ for any
$a_\hbar\in S^0(\R^2,\C)$. Indeed, again by Lemma \ref{lem:microloc}, one can assume that the symbol $a$ is supported on a small neighborhood of the curve $\mathcal{C}_\mu$. 
Then, letting $A_{ij} = \langle \phi_j , Op_{\hbar}(a) \phi_i \rangle$
for $i,j\in \N_0$, by definition of the Hilbert-Schmidt norm, 
\[
\|[\Pi_{\hbar},Op_{\hbar}(a)]\|_{J^2}^2 = \sum_{i,j\in\N_0} \big(\1\{j<N, i\ge N \}+ \1\{i<N, j\ge N \}\big) |A_{ij}|^2 ,
\]
where we recall that $N={\rm rank}(\1_{(-\infty,\mu]}(-\hbar^2\Delta+V))$.
By Proposition \ref{prop:limit_Toep}, the coefficients $A_{ij}$ decay like $|i-j|^{-\infty}$ so that we can truncate the previous sum to $N(\hbar)-\epsilon \le i,j\le N(\hbar)+\epsilon$ for any $\epsilon>0$, then by dominated convergence, we obtain
\[
\lim_{\hbar\to 0}\|[\Pi_{\hbar},Op_{\hbar}(a)]\|_{J^2}^2 =
\sum_{k\in \Z}|k||\hat{a}_k(\g(\mu))|^2.
\]

This proves the claim and we note that, by symmetry, the limit is $2\Sigma_{(V,\mu)}^2(a)$ according to formula ~\eqref{eq:Sigma}. 
We also record that the Szeg\H{o} limit in Theorem~\ref{thr:single_well} is $\Sigma_{(V,\mu)}^2(f)$ with $a = e^{\eta f}-1$. 
In particular,  as expected, we recover that $\Sigma_{(V,\mu)}^2(a) \sim \eta^2 \Sigma_{(V,\mu)}^2(f)$ as $\eta\to0$. 
Theorem \ref{thr:single_well} is a
generalisation of this $J_2$ norm convergence, by computing the limit of higher order terms (cumulants) in the Taylor expansion of the Laplace transform $\log
\mathbb{E}[e^{\eta\X(f)}]$, \eqref{Laplace}.

Alternative arguments (valid in any dimension) can be found in
\cite[Section~5.4]{deleporte_universality_2021} when $a$ is a function
of $x$ 
with compact support inside the bulk (by Lemma \ref{lem:microloc} we may assume that $a$ is supported on a
small neighbourhood of $C_{\mu}$) and also in the proof of \cite{fournais_optimal_2020}.
\end{proof}

\subsection{Fredholm determinants}
\label{sec:strong-szegho-limit}

The proof of Theorem \ref{thr:single_well} relies on a Taylor expansion of the Fredholm determinant  on the RHS of \eqref{Laplace}, or rather of a renormalised determinant (obtained by removing the mean of $\X(f)$); this requires still a few preparatory results.

Let $\Pi$ be a finite rank projection. 
Recall that if $A$ is an operator with $\|A\|<1$, then the Fredholm determinant
$\det(1+A\Pi)$ is well-defined and positive.
Moreover, the operator $\log(1+A)$ is bounded and we have
\begin{equation} \label{Fredholm_exp}
\Upsilon_{\Pi}(A) : =  \log \det(1+A\Pi)- \tr(\log(1+A)\Pi)=\sum_{n\geq
2}\frac{(-1)^n}{n}\tr(\Pi A^n\Pi-(\Pi A\Pi)^n).
\end{equation}
This series is (absolutely) convergent  since we have the trivial bound $\big\| (\Pi A\Pi)^n-\Pi A^n\Pi \big\|_{J^1} \le 2 N\|A\|^n$ where $N = \tr \Pi$. 
In particular, we write 
\(
\Upsilon_{\Pi}(zA) = \sum_{n\geq2} \tfrac{(-z)^n}{n} \Upsilon_{\Pi}^n(A)
\)
for $|z|<1$. 
Our next result provides a uniform bound for \eqref{Fredholm_exp} and a perturbative expansion of its coefficients in terms of the Hilbert-Schmidt norm of certain commutators. 

\begin{lem} \label{lem:Fredh}
For any operator $A$ with $\|A\|\le\rho$, one has $|\Upsilon_{\Pi}^n(A)| \le \tfrac{n(n-1)}{4} \| [\Pi,A] \|_{J^2}^2\, \rho^{n-2}$ for any $n \ge 2$. 
Moreover, if $\rho<1$, there is a constant $C$ depending only on $\rho$ and $\| [\Pi,A] \|_{J^2}$ so that
for any operator $R$ with $\|R\| \le\rho$, one has for any $n \ge 2$.
\[
|\Upsilon_{\Pi}^n(A+R) - \Upsilon_{\Pi}^n(A)| \le C 2^n  \max_{1\le k< n}  \| [\Pi,A^{k-1}R]\|_{J^2} . 
\]
\end{lem}

\begin{proof}
Observe that for $n\ge 2$,
\[
\Pi A^n\Pi - (\Pi A\Pi)^n = 
\tfrac12 \big( \Pi  [\Pi,A] [A^{n-1},\Pi]\Pi \big) + (\Pi A\Pi)(\Pi A^{n-1}\Pi - (\Pi A\Pi)^{n-1})
\]
so that by induction, 
\[\begin{aligned}
\| \Pi A^n\Pi - (\Pi A\Pi)^n \|_{J^1} 
&\le \tfrac12 \| [\Pi,A] [A^{n-1},\Pi]\|_{J^1} + \rho \|\Pi A^{n-1}\Pi - (\Pi A\Pi)^{n-1}\|_{J^1} \\
&\le \tfrac12 \| [\Pi,A] \|_{J^2} \sum_{1\le k<n} \rho^{k-1} \| [A^{n-k},\Pi]\|_{J^2}
\end{aligned}\]
where we used that $\|BQ\|_{J^1} \le \|B\|_{J^2}\|Q\|_{J^2}$  for Hilbert-Schmidt operators $B,Q$.
Moreover, we also have for $n\ge 2$,
\[
\| [A^{n},\Pi]\|_{J^2} \le \sum_{1\le k\le n} \| A^{k-1} [A,\Pi] A^{n-k}\|_{J^2} \le  n \rho^{n-1} \| [\Pi,A] \|_{J^2}  . 
\]
This implies that  for $n\ge 2$,
\[
\| \Pi A^n\Pi - (\Pi A\Pi)^n \|_{J^1} \le \tfrac{n(n-1)}{4} \| [\Pi,A] \|_{J^2}^2 \rho^{n-2} 
\]
which yields the first bound. 

For the second bound, we can expand for $n\ge 2$,
\begin{equation} \label{traceexpand}
\Upsilon_{\Pi}^n(A+R) - \Upsilon_{\Pi}^n(A) =  \sum_{\Gamma \in \Omega_n} \tr \big( \Pi\Gamma_1 \cdots \Gamma_n\Pi-\Pi\Gamma_1\Pi\cdots \Gamma_n\Pi  \big) .
\end{equation}
where $\Omega_n := \big\{ \Gamma \in\{A,R\}^n : \Gamma \neq (A,\cdots,A)  \big\}$. 
For $\Gamma \in \Omega_n$, let $\Gamma' \in \Omega_{n-1}$ be given by 
\[
\Gamma' = \begin{cases} 
(\Gamma_1,\dots, \Gamma_{n-1}) &\text{if } \Gamma_1=R \\
(\Gamma_2,\cdots,\Gamma_{n}) &\text{if } \Gamma_1=A 
\end{cases} .
\]
Let us also denote 
\[
\pi(\Gamma) = \Gamma_1 \cdots \Gamma_n
\qquad\text{and}\qquad 
{\rm Q}(\Gamma) = \Pi\Gamma_1 \cdots \Gamma_n\Pi-\Pi\Gamma_1\Pi\cdots \Gamma_n\Pi . 
\]

For $\Gamma \in \Omega_n$ with $n\ge 2$,  we claim that 
\begin{equation} \label{induction_Gamma}
\|{\rm Q}(\Gamma)\|_{J^1}  \le C \| [\Pi, \pi(\Gamma')] \|_{J^2} + \rho \|{\rm Q}(\Gamma')\|_{J^1}
\end{equation}
where $\| [A,\Pi] \|_{J^2} , \| [R,\Pi] \|_{J^2} \le   C$. 

To prove \eqref{induction_Gamma} in case $\Gamma_1 =R$, we decompose
\[ \begin{aligned}
Q(\Gamma) 
& = \Pi [\Pi,\Gamma_1 \cdots \Gamma_{n-1}]\Gamma_n\Pi
+\big(\Pi\Gamma_1 \cdots \Gamma_{n-1}\Pi-\Pi\Gamma_1\Pi\cdots \Pi\Gamma_{n-1}\Pi\big)\Pi\Gamma_n\Pi  \\
& = \Pi [\Pi, \pi(\Gamma')][\Gamma_n,\Pi] + Q(\Gamma') \Pi\Gamma_n\Pi . 
\end{aligned}\]
The argument is analogous in case $\Gamma_1 =A$. 

Then, by induction on  \eqref{induction_Gamma}, there is a sequence $\Gamma^k \in \Omega_k$
for $k =1,\cdots, n$ such that $\Gamma^{n} =\Gamma$,   $\Gamma^{k-1} =\Gamma^{k\prime}$ and 
\[
\|{\rm Q}(\Gamma)\|_{J^1}  \le C  \sum_{1\le k < n} \rho^{n-k-1}   \| [\Pi, \pi(\Gamma^k)] \|_{J^2} .
\]
Now we claim that there is a constant $c(\rho)$ so that for any $n\ge 2$ and $\Gamma \in \Omega_n$, 
\begin{equation} \label{induction_comm}
\| [\Pi, \pi(\Gamma)] \|_{J^2} \le c(\rho) \max_{0\le j<n} \| [A^j R,\Pi]\|_{J^2} .
\end{equation}
Using the two previous bounds, we conclude that there is a constant $C(\rho)$, so that for any $n\ge 2$ and $\Gamma \in \Omega_n$, 
\[
\|  \Pi\Gamma_1 \cdots \Gamma_n\Pi-\Pi\Gamma_1\Pi\cdots \Pi\Gamma_n\Pi \|_{J^1}
\le  C(\rho) \max_{0\le j<n-1} \| [A^j R,\Pi]\|_{J^2} .
\]
Going back to formula \eqref{traceexpand}, with $|\Omega_n| \le 2^n$,  this proves the claim.

Now, to prove \eqref{induction_comm}, note that for $\Gamma \in \Omega_n$, we can decompose
\[
\pi(\Gamma) = A^{\ell-1} R\, \Gamma_{\ell+1} \cdots \Gamma_n \qquad \text{for}\quad \ell \in \{1,\cdots, n\}
\]
so that 
\[ \begin{aligned}
[\Pi, \pi(\Gamma)] &  = [\Pi, A^{\ell-1} R] \Gamma_{\ell+1} \cdots \Gamma_n  
+ A^{\ell-1} R [\Pi,  \Gamma_{\ell+1} \cdots \Gamma_n ]  \\ 
\| [\Pi, \pi(\Gamma)] \|_{J^2} & \le \| [\Pi, A^{\ell-1} R] \|_{J^2} + \|  R [\Pi,  \Gamma_{\ell+1} \cdots \Gamma_n ] \|_{J^2} .
\end{aligned}\]
If 
$\Gamma_{\ell+1} =\cdots =\Gamma_n =A$, we can bound
\[ \begin{aligned}
R [\Pi,  \Gamma_{\ell+1} \cdots \Gamma_n ] &= [R ,\Pi] A^{n-\ell}  + [\Pi, RA^{n-\ell} ] \\
\|  R [\Pi,  \Gamma_{\ell+1} \cdots \Gamma_n ] \|_{J^2}  &\le  \|[R ,\Pi]\|_{J^2} + \| [\Pi, A^{n-\ell} R] \|_{J^2}
\end{aligned}\]
using that $ [\Pi,  R A^{\ell}]^* = - [\Pi, A^{\ell} R]$ for $\ell\ge 0$ since $A,R$ are self-adjoint; then we are done. 
Otherwise,
\[
\| [\Pi, \pi(\Gamma)] \|_{J^2} \le \| [\Pi, A^{\ell-1} R] \|_{J^2} + \rho \| [\Pi,  \pi(\Gamma') ] \|_{J^2} 
\qquad\text{where}\quad
\Gamma' \in \Omega_{n-\ell} , 
\]
and, by induction (using that $\rho<1$), we obtain \eqref{induction_comm}. 
\end{proof}

\begin{prop}  \label{prop:cutoff}
Define $\Upsilon_{\hbar}(a) := \Upsilon_{\Pi_\hbar}(Op_{\hbar}(a)) $ for any symbol $a_\hbar \in S^0(\R^2,\C)$ with $\|a_\hbar\|_{L^{\infty}}\le c <1$ as in Proposition~\ref{prop:Szegoasymp}. 
The coefficients $\Upsilon_{\hbar}^n(a)$ are bounded by constants $C_n$ (independent of $\hbar$) with $\sum_{n\ge 2} C_n <\infty$. 
Moreover, for any cutoff $\chi\in  C^\infty_c(\R^2,[0,1])$ which is equal to $1$ on a neighbourhood of the curve
$\mathcal{C}_{\mu}$, as $\hbar\to0$, 
\[
\Upsilon_{\hbar}(a) = \Upsilon_{\hbar}(a\chi) + o(1) 
\]
and $\Upsilon_{\hbar}^n(a) = \Upsilon_{\hbar}^n(a\chi)+ O_n(\hbar^\infty)$ for $n\ge 2$.
\end{prop}

\begin{proof}
Let $B=Op_{\hbar}(a)$, $A=Op_{\hbar}(\chi a)$, $R = Op_{\hbar}((1-\chi)a) $ and $\Pi=\Pi_\hbar$. By G\aa rding's inequality (\cite{zworski_semiclassical_2012}, Theorem 4.30)\footnote{Under our assumptions, $1-a \ge1-c>0$ so that  $Op_{\hbar}(1-a) = \mathrm{I}-B \ge \gamma \mathrm{I}$ as operator if $\hbar\le \hbar_\gamma$, for any $\gamma<1-c$.}, there exists a fixed $\rho<1$ such that for all $\hbar>0$ sufficiently small,
\[
\|A\|, \|B\| , \|R\| \le \rho.
\]
Moreover, $\| [\Pi,A] \|_{J^2}, \| [\Pi,B] \|_{J^2} \le C$ for a constant $C$ independent of $\hbar$ (by Lemma~\ref{prop:FM}), so using the first estimate from Lemma~\ref{lem:Fredh}, 
$|\Upsilon_{\hbar}^n(a)| \le C n^2 \rho^n$ and these constants are summable.

The, by Lemma~\ref{lem:microloc}, we claim that for any $k\in\N_0$, 
\begin{equation} \label{commest}
\| [\Pi,A^{k}R]\|_{J^2} =  \O_k(\hbar^\infty) . 
\end{equation}
Indeed, $A^kR=Op_{\hbar}(r_k)$ where $r_0 = (1-\chi)a$ and $r_{k,\hbar} = (\chi a)\# r_{k-1,\hbar}$ for $k\in\N$;  see \cite[Theorem 4.11]{zworski_semiclassical_2012}. 
Then, by \cite[Theorem 4.12]{zworski_semiclassical_2012}, $r_{k,\hbar} = r_{k,\hbar}' + \O_k(\hbar^\infty)$ where $\supp(r_{k,\hbar}') = \supp(r_0)$ for $k\in\N$ and since $ \supp(r_0)\cap \mathcal{C}_\mu = \emptyset$, this proves the estimate \eqref{commest}.  

Then, using the second bound from Lemma~\ref{lem:Fredh}, this implies that for any $n\ge 2$, 
\[
|\Upsilon_{\hbar}^n(a) - \Upsilon_{\hbar}^n(a\chi)|  = \O_n(\hbar^\infty) . 
\]
Hence, we also have $\Upsilon_{\hbar}(a) = \Upsilon_{\hbar}(a\chi) + o(1)$ as $\hbar\to0$ by the dominated convergence theorem. 
\end{proof}

Proposition~\ref{prop:cutoff} has the following important consequence; to compute the asymptotics of $\Upsilon_{\hbar}^n(a)$ for a general symbol $a_\hbar \in S^0(\R^2,\C)$, one can assume that $a_\hbar$ is supported in a small neighborhood of the curve $\mathcal{C}_{\mu}$.  
This assumption will be crucial to apply the \emph{quantum version of the action-angle theorem} in the next section. 
We can further replace the operator $Op_{\hbar}(\chi a)$ by a finite rank operator $A$ via a truncation. 

\begin{prop}  \label{prop:Fredh}
Let $a_\hbar \in S^0(\R^2,\C)$ be a symbol such that  $\|a_\hbar\|_{L^{\infty}}\le c <1$.
Let ${\rm Q}= \vartheta(-\hbar^2\Delta+V)$ and $\Pi = \1_{(-\infty,\mu]}(-\hbar^2\Delta+V)$ where $\vartheta\in C^\infty(\R,[0,1])$ is such that $\vartheta \ge \1_{(-\infty,\mu+\delta]}$ for a small $\delta>0$. 
Then, for any cutoff $\chi\in  C^\infty_c(\R^2,[0,1])$ supported on a
small neighbourhood of the curve $\mathcal{C}_{\mu}$ and equal to 1 on
a smaller neighbourhood of this curve, as $\hbar\to0$, 
\[
\Upsilon_{\hbar}(a) =  \log \det(1+Op_{\hbar}(a)\Pi)- \tr(\log(1+Op_{\hbar}(a))\Pi)  = \Upsilon_\Pi(A) + o(1) 
\]
where $A = {\rm Q}Op_{\hbar}(\chi a){\rm Q}$. 
Moreover,  $\Upsilon_{\hbar}^n(a) = \Upsilon_{\Pi}^n(A)+ O_n(\hbar^\infty)$ for any $n\ge 2$.
\end{prop}

\begin{proof} 
By Proposition~\ref{prop:cutoff}, on can replace the symbol $a$ by $\chi a$ in such a way that  $\supp(a\chi)\cap \mathcal{C}_{\mu\pm\delta} = \emptyset$.
Then, from the proof of Proposition~\ref{lem:microloc} (see \eqref{est1}), $\| [ {\rm Q},  Op_{\hbar}(a)] \|_{J^1} =  \O(\hbar^\infty)$.
This implies that for $n\ge 2$
\[
\Upsilon_\Pi^n(A)  = \tr(\Pi A^n\Pi-(\Pi A\Pi)^n) = \tr(\Pi Op_{\hbar}(a)^n\Pi-(\Pi Op_{\hbar}(a)\Pi)^n) + \O_n(\hbar^\infty)  
\]
where the leading term equals $\Upsilon_\hbar^n(a)$ and we used that $\Pi  {\rm Q} =   {\rm Q}\Pi  = \Pi$. 
Hence, using  again  that $\| Op_{\hbar}(a) \| \le \rho <1$ and $\| [\Pi,Op_{\hbar}(a)] \|_{J^2} \le C$ (Lemma~\ref{prop:FM}), by Lemma~\ref{lem:Fredh} and the dominated convergence Theorem, we conclude that
$\Upsilon_{\hbar}(a)   = \Upsilon_\Pi(A) + o(1) $ as $\hbar\to 0$ as claimed.
\end{proof}

\subsection{Strong Szeg\H{o} asymptotics}
\label{sec:end-proof}

Recall the notation of Section~\ref{sec:strong-szegho-limit} and let $a \in S^0(\R^2,\C)$ be any symbol with $\|a\|_{L^{\infty}}\le c <1$.
By Proposition~\ref{prop:Fredh}, we can also assume that $a$ is supported on a small neighbourhood of the curve $\mathcal{C}_{\mu}$ and it suffices to obtain the asymptotics of $\Upsilon_\Pi(A)$ where $A = {\rm Q}Op_{\hbar}(a){\rm Q}$. 
Computing the traces in question using the Hilbert basis $(\phi_\ell)_{\ell\in \N_0}$ from Proposition~\ref{prop:limit_Toep}, writing $A_{\ell k} = \langle\phi_\ell,Op_{\hbar}(a)\phi_k\rangle$, we obtain for $n\ge 2$, 
\begin{equation} \label{cumulant1}
\Upsilon_\Pi^n(A) = \big(\tr(\Pi A^n\Pi)-\tr(\Pi A\Pi)^n\big)
=  \sum_{j\in\N^{n+1} :  j_0= j_n <N} \omega(j) \,  A_{j_0j_1}A_{j_1j_2} \cdots A_{j_{n-1}j_n}
\end{equation}
where 
\(
\omega(j) = {\textstyle \prod_{k=1}^{n-1}}  \vartheta(\lambda_{j_k})^2(1- \1_{j_k<N})
\).
In particular, this sum is finite since $\vartheta(\lambda) = 0$ for $\lambda >\mu$ sufficiently large.
Here, we used that ${\rm Q} \phi_j = \vartheta(\lambda_j)\phi_j$ and
that all terms such that $ j_k <N$ for all $1\leq k\leq n$ cancel when computing the difference of both traces since $\vartheta(\lambda) = 1$ for $\lambda \le \mu$).

We can make a change of variables 
$ \{j\in\N^{n+1} : j_0= j_n \} \to  \{j_0 \in \N , i\in\Z^{n} : i_1+\cdots+ i_n = 0 \}$ given by 
$i_1=j_1-j_0,i_2=j_2-j_1,\ldots, i_{n-1}=j_0-j_{n-1}$, then the condition 
$\{ \exists 1\leq r \leq n-1 : j_r \ge N \}$ corresponds to 
\begin{equation} \label{conditionm}
\big\{ m_*(i) :=\max(i_1,i_1+i_2,\ldots,i_1+i_2+\ldots+i_{n}) \ge N-j_0 \big\} .
\end{equation}
Observe that if $|A_{jk}| \le \Gamma(k-j)$ for $\lambda_k,\lambda_j \in\supp(\vartheta)$, then we can bound for any $n\ge 2$, 
\[
\sum_{j\in\N^{n+1} :  j_0= j_n <N} \omega(j)  \big| A_{j_0j_1}A_{j_1j_2} \cdots A_{j_{n-1}j_n}\big| 
\le \sum_{i\in\Z^{n} : i_1+\cdots+ i_n = 0 } m_*(i)\Gamma(i_1) \cdots \Gamma(i_n) <\infty
\]
under the assumption that $\Gamma_k = \O(|k|^{-\infty})$ as $|k|\to\infty$. 
In particular, the \emph{cumulants} \eqref{cumulant1} are bounded (uniformly as $\hbar\to0$) in this case.

Then, we can apply the asymptotics of Proposition~\ref{prop:limit_Toep} in \eqref{cumulant1} (according to formula \eqref{Fourierb}, $|\hat{a}_k(I)| =  \O(|k|^{-\infty})$ as $|k|\to\infty$ uniformly for $I \in \R_+$), using the previous bound, 
we obtain 
\begin{equation} \label{cumulant2}
\Upsilon_\Pi^n(A) 
=  \sum_{j\in\N^{n+1} :  j_0= j_n <N} \omega(j) \,  
\hat{a}_{j_1-j_0}(\tfrac{j_1+j_0}2\hbar) \hat{a}_{j_2-j_1}(\tfrac{j_2+j_1}2\hbar) \cdots \hat{a}_{j_n-j_{n-1}}(\tfrac{j_n+j_{n-1}}2\hbar) +\O_n(\hbar) . 
\end{equation}
Moreover, we have $|\hat{a}_k(I)- \hat{a}_k(J)| =  \O\big(|I-J||k|^{-\infty}\big)$ 
as $|k|\to\infty$ uniformly for $I, J \in \R_+$ and also for any $j$ which contributes to \eqref{cumulant1}, 
\[
\big|\tfrac{j_s+j_r}2 - N \big| \le  m_*(i)+m_*(-i) , \qquad
\forall s,r \in [n]. 
\]
This allows to replace the arguments on the RHS of \eqref{cumulant2} by $N \hbar$, or equivalently by $I(\mu) \sim \hbar N$ as $\hbar\to0$  (by Weyl's law), up to a vanishing error as $\hbar\to0$, 
\[\begin{aligned}
\Upsilon_\Pi^n(A) 
& =  \sum_{j\in\N^{n+1} :  j_0= j_n <N} \omega(j) \,  
\hat{a}_{j_1-j_0}(I) \hat{a}_{j_2-j_1}(I) \cdots \hat{a}_{j_n-j_{n-1}}(I) +o_n(1)  \\
& =   \sum_{i\in\Z^{n} : i_1+\cdots+ i_n = 0 } m_*(i) \hat{a}_{i_1}(I) \cdots \hat{a}_{i_n}(I) +o_n(1)  .
\end{aligned}\]
Here, we used that according to \eqref{conditionm}, for every $i\in\Z^{n}$ with $i_1+\cdots+ i_n = 0$ and $m_*(i) \le \delta/\hbar$
\[
\omega(j) = \1\{m_*(i) \ge N-j_0\}
\]
and computed the sum over $j_0$ to obtain the last formula. 

Recall that the coefficients $\Upsilon_\Pi^n(A)$ are bounded uniformly in $\hbar$ in a summable way (by Proposition~\ref{prop:cutoff})
Hence, by the dominated convergence Theorem, this implies that 
\[
\Upsilon_\Pi(A) 
= \sum_{n\ge 2} \tfrac{(-1)^n}{n} \sum_{i\in\Z^{n} : i_1+\cdots+ i_n = 0 } m_*(i) \hat{a}_{i_1}(I) \cdots \hat{a}_{i_n}(I) + \underset{\hbar\to0}{o(1)} . 
\]

At this point, we can  apply the the Dyson-Hunt-Kac formula (\cite{Kac54}, Theorem 4.2),
which states that
\begin{align*}
\sum_{\sigma \in \mathfrak{S}_n} m^*(i_{\sigma(1)},\dots,i_{\sigma(n)})&= \sum_{\sigma\in
\mathfrak{S}_n} i_{\sigma(1)}\sum_{r=1}^n \1(i_{\sigma(1)}+\ldots+i_{\sigma(r)}>0)\\
&=\sum_{r=1}^n\frac{1}{2r}\sum_{\sigma\in
\mathfrak{S}_n}|i_{\sigma(1)}+\ldots+i_{\sigma(r)}|.
\end{align*}

By symmetry, this yields as $\hbar\to0$, 
\[
\Upsilon_\Pi(A) \sim \frac12 \sum_{n> r\ge 1} \frac{(-1)^n}{n \cdot r} \sum_{i\in\Z^{n} : i_1+\cdots+ i_n = 0 } |i_1+\cdots+i_r| \, \hat{a}_{i_1}(I) \cdots \hat{a}_{i_n}(I) .
\]
Grouping together terms with $i_1+\ldots+i_r=\ell$, we obtain
\[
\Upsilon_\Pi(A)  \sim \frac12 \sum_{n > r\ge 1} \frac{(-1)^n}{n \cdot r} \sum_{\ell\in\Z} |\ell| \, (\widehat{a^r})_\ell  (\widehat{a^{n-r}})_{-\ell} 
\]
where we used that $(\widehat{a^r})_\ell = \sum_{i \in \Z^r : i_1+\ldots+i_r=\ell} \hat{a}_{i_1} \cdots \hat{a}_{i_r}$ for $r\in\N$ and $\ell\in\Z$. 

By symmetry, we conclude that as $\hbar\to0$, 
\[ \begin{aligned}
\Upsilon_\Pi(A)  \sim \sum_{r, s \ge 1} \frac{(-1)^{r+s}}{2(r+s) \cdot r}  \sum_{\ell\in\Z} |\ell|(\widehat{a^r})_\ell  (\widehat{a^{s}})_{-\ell}
& = \frac12  \sum_{r, s \ge 1} \frac{(-1)^{r+s}}{s\cdot r} \sum_{\ell\in\Z} |\ell| (\widehat{a^r})_\ell  (\widehat{a^{s}})_{-\ell} \\
& =  \frac12 \sum_{\ell\in\Z} |\ell| \widehat{f}_\ell  \widehat{f}_{-\ell}
\end{aligned}\]
where
$f = -\log(1+a_0) =  \sum_{r\ge 1} \tfrac{(-1)^r}{r} a_0^r$ converges if $\|a_0\|<1$.

By Proposition~\ref{prop:Fredh}, this yields the asymptotics of Proposition~\ref{prop:Szegoasymp} and this completes the proof.

\subsection{Non-compactly supported test functions}

The technical condition $f\in C^{\infty}_c(\R,\R)$ is important for our proof of Theorem~\ref{thr:single_well}, but it can eventually be relaxed using the \emph{exponential decay} of the kernel of the spectral projector $\Pi_\mu$ in the \emph{forbidden region} $\{V>\mu\}$. 
However, the hypothesis  $\|a_\hbar\|_{L^\infty} \le c <1$ in Proposition~\ref{prop:Szegoasymp} is important to obtain the convergence of the Fredholm determinants (or the Laplace transform of the linear statistic $\X(f)$, \eqref{Laplace}). By relaxing the mode of convergence, we easily deduce the following result; 

\begin{prop}\label{prop:clt}
Let $f\in C(\R,\R)$ with at most exponential growth and such that
$f\in C^{\infty}(\{V<\kappa\})$ for some $\kappa>\mu$. Let $\mathcal{N}$
denote the standard Gaussian. Then, one has convergence in distribution
\[
\X(f) - \mathbb{E} \X(f) \, \to\,
\Sigma_{(V,\mu)}(f)\,\mathcal{N}\qquad \qquad \text{ as }\hbar \to 0
\]
and convergence of moments: for every $r\in\N$,
\[
  \mathbb{E}\big[\big( \X(f) - \mathbb{E} \X(f)\big)^r \big] \to
  \Sigma_{(V,\mu)}^r(f)\, \mathbb{E}[\mathcal{N}^r] \qquad \qquad
  \text{ as }\hbar\to0.
  \]
\end{prop}

\begin{proof}
Let $\chi \in C^\infty(\R,[0,1])$ be equal to $1$ on $\{V\le \kappa\}$. We claim that under the 
Assumptions of Proposition~\ref{prop:clt}, for any $r\in\N$,
\begin{equation} \label{trunc}
\mathbb{E}[\X(f)^{r}] = \mathbb{E}[\X(f\chi)^{r}]  + \O_r(\hbar^\infty) . 
\end{equation}
Indeed, recall that  $(\phi_j)_{j\in \N}$ is a Hilbert basis of eigenfunctions of $-\hbar^2\Delta+V$ with eigenvalues $(\lambda_j)_{j\in \N}$.
By \cite[Proposition 2.3]{deleporte_universality_2021}, there is a compact set
$\{V\le \mu\} \Subset  \mathcal{K} \Subset \{V\le \kappa\}$ and a constant $c>1$ such that
\[
\max_{\lambda_j\leq \mu}\big\|e^{\dist(x,  \mathcal{K}) /{\hbar}} \phi_j \big\|_{L^2}\leq c . 
\]
So, letting
$\mathcal{K}_k : = \big\{ x\in\R^n : k \le \dist(x,  \mathcal{K}) < k+1 \big\} $
for $k\ge 0$, we have 
\[
\max_{\lambda_j\leq \mu} \int_{\mathcal{K}_{k}^c} \phi_j^2 \le c e^{-k/{\hbar}} . 
\]
Moreover, since $f$ grows at most exponentially, there is a constant $C\ge 0$ so that $|f| \le e^{Ck/2}$ on $\mathcal{K}_k$. 
In particular, if $f=0$ on $\mathcal{K}_0$, then 
\[
\mathbb{E}[\X(f)^{2r}]  \le  \sum_{k \ge 1} e^{Crk} \mathbb{E}[\X(\mathds{1}_{\mathcal{K}_k})] , \qquad 
\mathbb{E}[\X(\mathds{1}_{\mathcal{K}_k})] \le \mathbb{E}[\X(\mathds{1}_{\mathcal{K}_{k-1}^c})] = \int_{\mathcal{K}_{k-1}^c} \hspace{-.3cm}\Pi_\hbar(x,x) \dd x  \le N \max_{\lambda_r \le \mu} \int_{K_{j-1}^c} \phi_r^2
\]
where ${\mathcal{K}_k^c}=\R^n\setminus \mathcal{K}_k$ for $k\ge 0$. 
Since $N \asymp \hbar^{-n}$, this implies that 
\[
\mathbb{E}[\X(f)^{2r}]  \le \O(\hbar^{-n}) \sum_{k \ge 1} e^{Crk-k/\hbar} = \O_r(\hbar^\infty)  
\]
since this geometric sum is exponentially small as $\hbar\to0$. 
This proves \eqref{trunc} since we can choose a cutoff  $\chi =1$ on~$\mathcal{K}_0$. 

Hence, using \eqref{trunc}, it follows immediately from Theorem~\ref{thr:single_well} that in distribution,
\[
\X(f) - \mathbb{E} \X(f) \, \to\,  \Sigma_{(V,\mu)}(\chi f)\,\mathcal{N}
\] 
and all moments converge. Since, $\chi=1$ on  $\{V\le \mu\}$, the
Fourier coefficients of the functions $\chi f$ and $f$ coincide in
formula \eqref{eq:Sigma} so that $\Sigma_{(V,\mu)}(\chi f) =
\Sigma_{(V,\mu)}(f)$ as expected. This concludes the proof.
\end{proof}

\section{Multi-cut case -- Proof of Theorem \ref{thr:multiple_wells}}
\label{sec:several-wells}

The action-angle theorem (Proposition \ref{prop:aa}) has a microlocal
nature, and allows to find quasimodes for any Schrödinger operator
using only a local hypothesis. Thus, given $V:\R\to \R$
confining such that the support of the density of states $\{V\leq \mu\}$ consists of several disjoint (compact) intervals,
each satisfying the geometric requirements of Section
\ref{sec:one-well-case}, the eigenvalues (up to $\mu$) of $-\hbar^2\Delta+V$
are given, up to a small error, by the union of the eigenvalues of
local models $-\hbar^2\Delta+V_j$, where the potential $V_j$ has a
single well; the associated quasimodes will be localised on a single
well. Under a generic condition, the eigenvalues of different wells
are sufficiently separated from each other, so that the quasimodes are very close
to actual eigenfunctions.

\begin{prop}\label{prop:decomposition}(\cite{helffer_multiple_1984}, Lemma 2.3 and Theorem 2.4)
Let $V\in L^1_{\rm loc}(\R,\R)$ and let $\mu>0$. Suppose that, for
some $\epsilon>0$, the set $\{V\leq
\mu+\epsilon\}$ (consisting of disjoint intervals $I_1',\ldots, I_\ell'$) is compact, on which $V$ is $C^{\infty}$. Suppose
also that $\{V\leq \mu\}$ is a finite union of non-empty intervals $I_1,\ldots, I_\ell$ with $I_j\subset I_j'$ for $j=1,\dots,\ell$.

For $j=1,\dots,\ell$, let $V_j\in C^{\infty}(\R,\R)$ be equal to
$V$ on $I_j'$, and greater than $\mu+\epsilon$ on $\R \setminus I_{j'}$. For all $\hbar>0$ and
$1\leq j\leq \ell$, consider the eigenpairs $(\phi_{j,k},\lambda_{j,k})$ of
$-\hbar^2\Delta+V_j$ such that $\lambda_{j,k}\leq
\mu+\epsilon/2$.
\begin{enumerate}
\item There exists $c>0,C>0$ such that
\[
\|(-\hbar^2\Delta-V-\lambda_{j,k})\phi_{j,k}\|_{L^2}\leq
Ce^{-c\hbar^{-1}}.
\]
\item Letting $E={\rm span}(\phi_{j,k} : \lambda_{j,k}\leq
\mu+\epsilon/2 , j=1,\dots,\ell )$, for all $u\in
E^{\perp}$,
\[
\langle u,(-\hbar^2\Delta-V),u\rangle\geq
(\mu+\tfrac{\epsilon}{4})\|u\|^2_{L^2}.
\]
\end{enumerate}
\end{prop}

Thus, the eigenvalue/eigenfunction asymptotics in the case where $V$ has
several wells can be reduced to the one-well case if the eigenvalues $\{\lambda_{j,k}\leq
\mu+\epsilon/2 \}$ are well-separated. Indeed, under such assumptions, the eigenfunctions $\{\phi_{j,k}\}$ are quasi orthogonal and ultimately one controls the asymptotics of the spectral projector $\Pi_\hbar(\mu)$. Using usual
results of spectral stability for self-adjoint operators
(\cite[Theorem 8.20]{helffer_spectral_2013}, see also footnote \ref{fn:stability}), we obtain the following result.

\begin{prop}\label{prop:separation_evs}
Under the assumptions of Proposition~\ref{prop:decomposition}, let
\[
\Pi_{\hbar;j}=\1_{(-\infty,\mu]}(-\hbar^2\Delta+V_j) , \qquad 1\leq
j\leq \ell . 
\]
Suppose that the  eigenvalues are not exponentially close to each
other or to $\mu$: there is a small $\epsilon>0$ so that for all $\hbar$ sufficiently small, 
\begin{enumerate}
\item for every $1\le j\neq k \le \ell$, $[\mu-\epsilon,\mu+\epsilon]\cap {\rm Sp}(-\hbar^2\Delta+V_j)$
lies at distance greater than $e^{-\epsilon\hbar^{-1}}$ from
$[\mu-\epsilon,\mu+\epsilon]\cap {\rm Sp}(-\hbar^2\Delta+V_k)$.  
\item for all $1\leq k\leq \ell$, $[\mu-e^{-\epsilon\hbar^{-1}},\mu+e^{-\epsilon\hbar^{-1}}]\cap {\rm Sp}(-\hbar^2\Delta+V_k) = \emptyset$. 
\end{enumerate}
Then, the projector
\begin{equation} \label{spec_decomp}
\Pi_{\hbar}=\sum_{j=1}^\ell\Pi_{\hbar;j}+\O_{J^1}(\hbar^{\infty}),
\end{equation}
and, for any symbol $a\in S^0(\R^2,\R)$ with $\|a\|_{L^\infty}<1$,
\[
\log \det(I+Op_{\hbar}(a)\Pi_{\hbar})=\sum_{j=1}^\ell\log
\det(I+Op_{\hbar}(a)\Pi_{\hbar;j})+\O(\hbar^{\infty}).
\]
\end{prop}
\begin{proof}
First, we  can isolate $[\mu- \frac\epsilon2,\mu+\frac\epsilon2]$ from the
rest of the spectrum. Let $\chi\in C^{\infty}(\R,[0,1])$ be equal to
$1$ on a neighbourhood of $(-\infty,\mu-\frac\epsilon4]$ and equal to
$0$ on a neighbourhood of $[\mu,+\infty)$. Decompose
\[
\Pi_{\hbar}=\chi(H_{\hbar})+(1-\chi)\1_{(-\infty,\mu]}(H_{\hbar})
\]
and decompose similarly $\Pi_{\hbar;j}$ for $1\leq j\leq \ell$. By
the smooth functional calculus of pseudo-differential operators
(\cite{dimassi_spectral_1999}, Theorem 8.7),
\[
\chi(H_{\hbar})=\sum_{1\leq j \leq \ell}\chi(H_{\hbar;j})+\O_{J^1}(\hbar^{\infty})
\]
since both sides are pseudo-differential operators with the same
symbol at any order.

It remains to study the spectrum in $[\mu- \frac\epsilon2,\mu+\frac\epsilon2]$, where we use the hypothesis of separation of eigenvalues.
For every $1\leq j\leq \ell$, let $\chi_j\in C^{\infty}_c(\R,[0,1])$
 with $\1_{I_j} \le \chi_j \le \1_{I_j'}$ (see Proposition~\ref{prop:decomposition}). 
By Agmon estimates \cite{simon_semiclassical_1984}, there is a $c>0$
so that if $u_{\hbar}$ is a (normalised) eigenfunction of $H_{\hbar}$
with energy $\lambda\in [\mu- \frac\epsilon2,\mu+\frac\epsilon2]$, for
every $1\leq j\leq\ell$, $\chi_ju_{\hbar}$ is an $O(e^{-c\hbar^{-1}})$-quasimode for both
$H_{\hbar}$ and $H_{\hbar;j}$ at energy $\lambda$ and
$\sum_{j\leq \ell}\|\chi_ju_{\hbar}\|_{L^2}=1-O(e^{-c\hbar^{-1}})$.
Moreover, by Remark~\ref{rk:eigsep}, for every $1\leq j\leq \ell$, the eigenvalues of $H_{\hbar;j}$ are simple and separated by about $\hbar$.
 Hence, by spectral stability (\cite[Theorem 8.20]{helffer_spectral_2013}, see also footnote~\ref{fn:stability}), if $\|\chi_ju_{\hbar}\|_{L^2}\geq \tfrac{1}{2\ell}$, there is an eigenvalue $\lambda_{j,k}$ of $H_{\hbar;j}$ such that  
$|\lambda- \lambda_{j,k}| =\O(e^{-c\hbar^{-1}/2})$. 
 Using the separation hypothesis~1. with $\epsilon<c/2$, we conclude
 that there is exactly
 one $1\leq j\leq \ell$ such that $\|\chi_ju_{\hbar}\|_{L^2}\geq \tfrac{1}{2\ell}$, 
 and that $\chi_ju_{\hbar} = \phi_{j,k}+ O(e^{-c\hbar^{-1}/2})$, where
 $(\phi_{j,k})_{k\in\N}$ denotes a collection of normalised
 eigenfunctions of $H_{j,\hbar}$ with increasing eigenvalues $\lambda_{j,k}$. This has the following consequences;
\vspace{-.2cm}
\begin{itemize}[leftmargin=*,itemsep=-.3em]
\item[(a)] The spectrum of $H_{\hbar}$ in $[\mu-
  \frac\epsilon2,\mu+\frac\epsilon2]$ is simple and its eigenvalues
  are separated by about $e^{-\epsilon\hbar^{-1}}$. 
\item[(b)] For every eigenpair $(u_{\hbar},\lambda)$ of $H_{\hbar}$, there is a unique $(j,k)$
such that $|\lambda- \lambda_{j,k}| =\O(e^{-c\hbar^{-1}/2})$ and $u_{\hbar} = \phi_{j,k}+ O(e^{-c\hbar^{-1}/2})$.  
\end{itemize}
\vspace{-.2cm}
On the other-hand, by Proposition~\ref{prop:decomposition}, the $\phi_{j,k}$ are also $O(e^{-c\hbar^{-1}})$-quasimodes of $H_{\hbar}$ with energy $\lambda_{j,k}$ so that using that the eigenvalues of $H_{\hbar}$ are separated (a), there is a 1-1 correspondence between $\operatorname{spec}(H_{\hbar})\cap [\mu- \frac\epsilon2,\mu+\frac\epsilon2]$ and the relevant part of $\bigcup_{j=1}^\ell \operatorname{spec}(H_{\hbar; j})$ as in (b).

Now, (a) and the separation hypothesis~1. with $\epsilon<c/2$ also guarantees that 
 $\operatorname{spec}(H_{\hbar})\cap [\mu- \frac\epsilon2,\mu]= \bigcup_{j=1}^\ell \operatorname{spec}(H_{\hbar; j})\cap [\mu- \frac\epsilon2,\mu]$ and using (b), we obtain the spectral decomposition; 
 \[
(1-\chi)\1_{(-\infty,\mu]}(H_{\hbar})=\sum_{j=1}^{\ell}(1-\chi)\1_{(-\infty,\mu]}(H_{\hbar;j})+\O_{J^1}(Ne^{-c/2\hbar}),
\]
which concludes this part of the proof.

To prove closeness of the $\log$-determinants, we write $A=Op_{\hbar}(a) \sum_{j=1}^\ell
\Pi_{\hbar,j}$, $B=Op_{\hbar}(a)\Pi_{\hbar}=A+R$, so that
\[
\det(I+B)=\det(I+A)\cdot\det(I+R(1+A)^{-1})
\]
and $\|R(1+A)^{-1}\|_{J^1}=\O(\hbar^{\infty})$ since $\|R\|_{J^1}=\O(\hbar^{\infty})$ and $(1+A)^{-1}$ is a 
bounded (by assumption $\|a\|_{L^\infty}<1$ and by G\aa rding's inequality $\|Op_{\hbar}(a)\| <1$).  
By continuity of $\log\det$ with respect to the $J^1$ norm, we conclude that
\[
\log\det(I+B)=\log\det(I+A) + \O(\hbar^{\infty}). 
\]
Now, as the projections satisfy
$[\Pi_{\hbar;j},\Pi_{\hbar;j'}]= \O(\hbar^\infty)$ for all $1\leq
j,j'\leq \ell$, using the Baker-Campbell-Haussdorf formula \cite{dynkin_representation_1949}, 
\[
\log\det(I+A) =  \sum_{j=1}^{\ell} \log\det(I+Op_{\hbar}(a)\Pi_{\hbar;j}) + \O(\hbar^{\infty})
\]
as claimed.
\end{proof}

We are now ready to complete the proof of Theorem
\ref{thr:multiple_wells} by a Borel-Cantelli argument.

As in the hypothesis of Theorem \ref{thr:multiple_wells}, let $(\chi_k)_{1\leq k\leq \ell}$ be a family of $C^\infty_c(\R,\R_+)$ functions with disjoint supports such that $\chi_k=1$ on $I_k$ and consider the Schr\"odinger  operator
\begin{equation}\label{We}
-\hbar^2_N\Delta+W^w , \qquad W^w:= V+{\textstyle \sum_{j=1}^\ell w_j\chi_j}
\end{equation}
for $w\in [-\epsilon,\epsilon]^\ell$ where $\epsilon>0$ is a small parameter. 
In particular, for almost every $w\in [-\epsilon,\epsilon]^\ell$, the set $\{W^w \le \mu\}$  still consists of $\ell$ disjoint (compact) intervals, denoted $I_1^w,\ldots, I_\ell^w$ and $\nabla W^w= \nabla V \neq 0$ on $\{W^w = \mu\}$. 

Replacing $V$ by $W^w$ amounts to replacing $V_j$ with $W^w_j=V_j+w_j$ (so that the eigenvalues
$\lambda_{k;j}$ change into  $\lambda_{k;j}+w_j$ for $1\leq j\leq \ell$) in Propositions~\ref{prop:decomposition} and~\ref{prop:separation_evs}. 
Then, observe that for $\hbar>0$, the measure of the set 
\[\big\{ w\in [-\epsilon,\epsilon]^\ell : \exists (k,j) , (k',j') \text{ with } j\neq j' \text{ and } |\lambda_{k;j}-\lambda_{k';j'}+w_j-w_j'|<e^{-c/2\hbar}
\text{ or }   |\lambda_{k;j}+w_j-\mu|<e^{-c/2\hbar}\big\}\] is $\O(\hbar^{-2}e^{-c/2 \hbar})$ where the factor $\O(\hbar^{-2})$ accounts for all possible choices of $k,k'$. 
Hence, for any sequence $\hbar_m \to 0$ as $m\to\infty$ sufficiently fast so that 
$\sum_m \hbar_m^{-2}e^{-c/2\hbar_m}<\infty$ (this is plainly the case if for some $\alpha>0$, $\hbar_m \leq Cm^{-\alpha}$), by the Borel-Cantelli Lemma, there is a full measure set $E \subset  [-\epsilon,\epsilon]^\ell$ such that for $w\in E$, if $m$ is sufficiently large, the operator \eqref{We} satisfies both assumptions of Proposition \ref{prop:separation_evs}. 
Hence, applying this Proposition with $\Pi_{m}=\1_{(-\infty,\mu]}(-\hbar^2_m\Delta+W^w)$ and 
$\Pi_{m,j} =\1_{(-\infty,\mu]}(-\hbar^2_m\Delta+W^w_j)$ for $1\leq
j\leq \ell$,  we conclude that for any symbol $a\in S^0(\R^2,\R)$ with $\|a\|_{L^\infty}<1$, as $m\to\infty$, 
\[\begin{aligned}
\log \det(I+Op_{\hbar}(a)\Pi_{m}) 
&= \sum_{j=1}^{\ell} \log\det(I+Op_{\hbar}(a)\Pi_{m;j})  + o(1) \\
& = \sum_{j=1}^{\ell} \bigg( \tr(\log(1+Op_{\hbar}(a))\Pi_{m;j})
+ \frac{1}2  \sum_{\ell\in\Z} |\ell| \widehat{f}_\ell(I_j)  \widehat{f}_{-\ell}(I_j)\bigg)+o(1)
\end{aligned}\]
where we applied the (one-cut) asymptotics of Proposition~\ref{prop:Szegoasymp} at the second step, $f = -\log(1+a)$ and $I_j = \g^{-1}(\mu-w_j)$ according to the notation of Section~\ref{sec:notations}. 
For a given $f\in C^{\infty}_c(\R,\R)$, choosing the symbol $a= e^{\eta f}-1$  which is in $S^0(\R^2,\C)$ with $\|a\|_{L^\infty}<1$ if the parameter $\eta\in\C$ is sufficiently small, this completes the proof of Theorem \ref{thr:multiple_wells}.
Here, we have $\tr(f\Pi_{m}) =  \sum_{j=1}^\ell \tr(f\Pi_{m;j}) + o(1)$
(by Proposition \ref{prop:separation_evs}) and the variance equals
$\sum_{j=1}^{\ell} \Sigma_{(W^w_j,\mu)}^2(f)$. This computes the proof.

\appendix
\section{Gaussian Free field (GFF) interpretation}\label{sec:GFF}

We provide a functional interpretation of central limit Theorem~\ref{thr:single_well} in terms of the Gaussian free field. 
This interpretation of Szeg\H{o}-type limit theorems is classical in random matrix theory, pioneered in \cite{HKO01} for the circular unitary ensemble (Example~\ref{ex:cue} -- we refer~e.g.~\cite{CFLW21} for recent developments), this proceeds by considering the counting function
\[
h_N :  x\mapsto \X(\1_{(-\infty,x]}) . 
\]
In the physics literature, the asymptotics of the correlation kernel of $h_N$ were recently obtained using  WKB asymptotics of the Schr\"odinger operator eigenfunctions and the result is interpreted in terms of the GFF, though Gaussian fluctuations were not established in this paper. 

The connection with Theorem~\ref{thr:single_well} comes from the fact that for a Schwartz function $f\in \mathcal{S}$, 
\[
\X(f) = - \int f'(x) h_N(x) \dd x 
\]
so that, viewed as a random Schwartz distribution, $\widetilde{h_N} : = \sqrt 2\pi\big( h_N  - \mathbb{E} h_N\big)$ converges weakly as $N\to\infty$ (equivalently  $\hbar\to 0$) to a (centred) Gaussian random field $h \in \mathcal{S}'$ with covariance kernel; 
\begin{equation} \label{corrkernel}
 \widetilde{\mathcal H} : (x,z) \mapsto  \big(\log\big|\sin \tfrac{\theta(x)+\vartheta(z)}{2}\big| - \log\big|\sin \tfrac{\theta(x)-\vartheta(z)}{2}\big|\Big) ,\qquad  x,z \in [x_0^-,x_0^+]
\end{equation}
where recall that $x_0^{\pm}(\mu)$ are such that $\{V<\mu\} =(x_0^-,x_0^+)$, $T=T(\mu)$ is the period of the flow \eqref{Hflow} with $(x_0,\xi_0) = (x_0^+,0)$ and the map 
\begin{equation} \label{theta}
\theta(x) =  \frac{\pi}{T} \int_x^{x_0^+} \frac{\dd u}{\sqrt{\mu-V(u)}} , \qquad x\in [x_0^-,x_0^+] . 
\end{equation}
Formula \eqref{corrkernel} is (up to normalization) \cite[Formula (16)]{Smith_21} and it has the following interpretation. 

Let $\xi$ be the GFF on $\mathbb{U} = \{z\in\C : |z| =1\}$, that is, the restriction of a 2d GFF on the unit circle. Suitably normalized, $\xi$ is a log-correlated field with covariance kernel
\[
\mathbb E\big[\xi(z)\xi(w)\big] = \log|1-z\overline{w}|^{-1} , \qquad z, w\in \mathbb{U} . 
\]
Then, \eqref{corrkernel} is the kernel of the pushforward of the GFF by the map \eqref{theta} in the following sense; 
\[
 \widetilde{\mathcal H} (x,z) = \mathbb E h(x)h(z) , \qquad h(x) = \frac{\xi(e^{\i\theta(x)})-\xi(e^{-\i\theta(x)})}{\sqrt 2} . 
\]
Note that the normalization of $\widetilde{h_N}$ is similar to \cite{CFLW21}  (GUE case) and the limit $h$ is a standard log-correlated field. 

\smallskip

In the remainder of this section, we explain how to obtain \eqref{corrkernel} from formula \eqref{var} and the consideration from Section~\ref{sec:notations}. 
Let $\Sigma = \Sigma_{(V,\mu)}^2$, by Devinatz formula \cite[Proposition 6.1.10]{simon_orthogonal_2005}, 
\begin{equation} \label{vardev}
\begin{aligned}
\Sigma(f) &= \frac12
\iint_{[-\pi,\pi]^2} \bigg| \frac{ f(\psi(\theta))- f(\psi(\vartheta))}{e^{\i\theta}-e^{\i\vartheta}} \bigg|^2
\frac{d\theta}{2\pi}\frac{d\vartheta}{2\pi} \\
& = \iint_{[0,\pi]^2} \big( f(\psi(\theta))- f(\psi(\vartheta))\big)^2\bigg(\frac1{|1-e^{\i(\theta-\vartheta)}|^2}+\frac1{|1-e^{\i(\theta+\vartheta)}|^2}\bigg)
\frac{d\theta}{2\pi}\frac{d\vartheta}{2\pi} .
\end{aligned}
\end{equation}
Here, we used that the curve $\mathcal{C}_\mu$ is symmetric with respect to the axis $\{\xi=0\}$ and
the map $\psi(\theta) = x_{\frac{\theta T}{2\pi}}$ in terms of the flow \eqref{Hflow} so that 
$\psi(2\pi-\theta) =\psi(\theta)$ for $\theta\in[0,\pi]$. 
We can invert the map $\psi : (0,\pi) \to (x_0^-,x_0^+)$ using \eqref{Hflow}  to compute its derivative; 
\[
\psi'(\theta) = \tfrac{T}{\pi} \xi_{\frac{\theta T}{2\pi}} = \tfrac{T}{\pi}\sqrt{\mu- V(\psi(\theta))} , \qquad \theta\in(0,\pi) . 
\]
In particular, the period of the flow is $T(\mu) = \int_{x_0^-}^{x_0^+} \frac{d x}{\sqrt{\mu-V(x)}} $ and the inverse map is given by \eqref{theta} using that $\theta'(x) = 1/{\psi'(\theta(x))}$ for  $x\in(x_0^-,x_0^+)$. 
The kernel on the RHS of \eqref{vardev} is 
\[\begin{aligned}
\mathcal{K} : (\theta,\vartheta) \mapsto & \frac1{|1-e^{\i(\theta-\vartheta)}|^2}+\frac1{|1-e^{\i(\theta+\vartheta)}|^2} = \frac{1}{2}\frac{2-\cos(\theta-\vartheta)-\cos(\theta+\vartheta)}{(1-\cos(\theta-\vartheta))(1-\cos(\theta+\vartheta))}\\
& = \frac{1-\cos\theta \cos\vartheta}{(\cos\theta - \cos\vartheta)^2} . 
\end{aligned}\] 
Hence, by a change of variables, we obtain an equivalent formula for the variance
\[\begin{aligned}
\Sigma(f) 
& = \frac{1}{(2\pi T)^2}\iint_{\{V<\mu\}^2} \ \big( f(x)- f(z)\big)^2
\frac{1-\cos\theta(x) \cos\theta(z)}{(\cos\theta(x) - \cos\theta(z))^2}
\frac{\dd x}{\sqrt{\mu-V(x)}}\frac{\dd z}{\sqrt{\mu-V(z)}} . 
\end{aligned}\]

To obtain \eqref{corrkernel}, we observe that 
\[
\mathcal{K} : (\theta,\vartheta) \mapsto -\partial_\theta \partial_\vartheta \mathcal{H}(\theta,\vartheta) ,
\qquad
\mathcal{H} : (\theta,\vartheta) \mapsto  \big(\log\big|\sin \tfrac{\theta+\vartheta}{2}\big| - \log\big|\sin \tfrac{\theta-\vartheta}{2}\big|\Big)
\]
and that 
\(
 \big[\partial_\vartheta \mathcal{H}(\theta,\vartheta)\big]_{\theta =0}^\pi =0 , 
 \)
\(
 \big[\mathcal{H}(\vartheta,\theta)\big]_{\theta =0}^\pi =0 .
 \)
Hence, integrating by parts twice, 
\[\begin{aligned}
\Sigma(f) 
&= -  \iint_{[0,\pi]^2} \partial_\theta \partial_\vartheta \big( f(\psi(\theta))- f(\psi(\vartheta))\big)^2 \mathcal{H}(\theta,\vartheta)
\frac{d\theta}{2\pi}\frac{d\vartheta}{2\pi} \\
&= 2 \iint_{[0,\pi]^2}  f'(\psi(\theta))f'(\psi(\vartheta)) \mathcal{H}(\theta,\vartheta) \frac{d\psi(\theta)}{2\pi}\frac{d\psi(\vartheta)}{2\pi}  \\
& = \frac{1}{2\pi^2} \iint_{\{V<\mu\}^2} f'(x) f'(z) \widetilde{\mathcal H}(x,z) dx dz 
\end{aligned}\]
with the kernel $\widetilde{\mathcal H}$ as in \eqref{corrkernel}. 
Hence, according to Theorem~\ref{thr:single_well}, we conclude that for a $f\in \mathcal{S}$, it holds in distribution
\[
\sqrt{2}\pi \int f'(x) h_N(x) \dd x  \, \to\, \int f'(x) h(x) \dd x . 
\]

\bigskip
\noindent
{\bf Acknowledgement.}
We thank the authors of \cite{Smith_21} for pointing out their work on the counting function of 1d Fermi gas in the one-cut case and the connection with our Theorem~\ref{thr:single_well}.

\end{document}

%% file: thm-config.tex
\usepackage{amsthm}
\usepackage{amssymb}
\usepackage{xparse}
\usepackage{thmtools}
\usepackage{stackrel}
\usepackage{tikz}

\usetikzlibrary{hobby}

\newcommand{\R}{\mathbb{R}}
\newcommand{\C}{\mathbb{C}}
\newcommand{\Z}{\mathbb{Z}}
\newcommand{\N}{\mathbb{N}}

\newcommand{\dd}{\mathrm{d}}

\renewcommand{\O}{\mathcal{O}}

\DeclareMathOperator{\supp}{supp}

\DeclareMathOperator{\tr}{tr}

\DeclareMathOperator{\dist}{dist}

\makeatletter
\newtheoremstyle{indented}
{7pt} 
{7pt} 
{} 
{1.5em} 
{\bfseries} 
{.} 
{.5em} 
{} 

\theoremstyle{definition}

\newtheorem{defn}{Definition}[section]

\newtheorem{example}[defn]{Example}
\theoremstyle{plain}
\newtheorem*{theorem*}{Theorem}

\newtheorem{theorem}{Theorem}

\newtheorem{prop}[defn]{Proposition}

\newtheorem{lem}[defn]{Lemma}

\theoremstyle{definition}
\newtheorem{rem}[defn]{Remark} 

\makeatletter
\renewcommand*\env@matrix[1][*\c@MaxMatrixCols c]{%
  \hskip -\arraycolsep
  \let\@ifnextchar\new@ifnextchar
  \array{#1}}
\makeatother
